\documentclass[12pt,oneside,a4paper]{amsart}
\usepackage{comment}
\usepackage[T1]{fontenc}
\usepackage{fullpage}
\usepackage{mathtools}
\usepackage[english]{babel}
\usepackage[utf8]{inputenc}
\usepackage{graphicx}
\usepackage{mathcomp}
\usepackage{amsfonts}
\usepackage{amssymb}
\usepackage{amsmath}
\numberwithin{equation}{section}\swapnumbers
\usepackage{pdflscape}
\usepackage{booktabs}
\usepackage{relsize}
\usepackage{footnote}
\usepackage{longtable}
\usepackage{xcolor}
\usepackage{pict2e}
\usepackage{nicefrac}
\usepackage{array}
\newcolumntype{D}{>{\setbox0=\hbox\bgroup}c<{\egroup}@{}}
\newcolumntype{d}{>{\setbox0=\hbox\bgroup$}c<{$\egroup}@{}}
\newcolumntype{Z}{>{\setbox0=\hbox\bgroup}c<{\egroup}@{\hspace*{-\tabcolsep}}}
\usepackage[arrow, matrix, curve]{xy}
\usepackage{mathtools}
\usepackage{tikz}
\usetikzlibrary{shapes.geometric, arrows}
\usetikzlibrary{chains}
\usetikzlibrary{backgrounds}
\usepackage{blkarray}
\usepackage{multirow}
\usepackage{lunadiagrams-KM-4}
\usepackage[colorlinks,
pdfpagelabels,
pdfstartview = FitH,
bookmarksopen = true,
bookmarksnumbered = true,
linkcolor = blue,
urlcolor = blue,
plainpages = false,
hypertexnames = false,
citecolor = blue] {hyperref}
\usepackage{enumitem}
\setlist[enumerate,2]{label=\textit{\alph*)},ref=\textit{\alph*})}
\setlist[enumerate,1]{label=\textit{(\arabic*)},ref=\textit{(\arabic*)}}
\usepackage[capitalize]{cleveref}
\tikzstyle{startstop} = [rectangle, rounded corners, minimum width=3cm, minimum height=1cm,text centered, draw=black, fill=red!30]
\tikzstyle{io} = [rectangle, trapezium left angle=70, trapezium right angle=110, minimum width=3cm, minimum height=1cm, text centered, draw=black, fill=blue!30]
\tikzstyle{process} = [rectangle, minimum width=3cm, minimum height=1cm, text centered, draw=black, fill=orange!30]
\tikzstyle{decision} = [rectangle, minimum width=1cm, minimum height=1cm, text centered, draw=black, fill=green!30]
\tikzstyle{arrow} = [thick,->,>=stealth]
\renewcommand{\[}{\begin{equation}}
\renewcommand{\]}{\end{equation}}
\theoremstyle{plain}    
\newtheorem{thm}{Theorem} 
\newtheorem{lem}[thm]{Lemma}
 \newtheorem{cor}[thm]{Corollary}
 
 \theoremstyle{remark}
 \newtheorem{rem}[thm]{Remark}
 \newtheorem{rems}[thm]{Remarks}
 \theoremstyle{definition}
 \newtheorem{defn}[thm]{Definition}
  \newtheorem{example}[thm]{Example}
  \newtheorem{lis}[thm]{Table}
  \numberwithin{thm}{section}
\newcommand{\Aut}{\operatorname{Aut}}
\newcommand{\sA}{{\mathsf A}}
\newcommand{\sB}{{\mathsf B}}
\newcommand{\sC}{{\mathsf C}}
\newcommand{\sD}{{\mathsf D}}
\newcommand{\sE}{{\mathsf E}}
\newcommand{\sF}{{\mathsf F}}
\newcommand{\sG}{{\mathsf G}}
\newcommand{\cD}{{\mathcal D}}
\newcommand{\rk}{\operatorname{rk}}
\newcommand{\aaa}{\mathfrak{a}}
\newcommand{\Ca}{{\bf a}}
\newcommand{\Cc}{{\bf c}}
\newcommand{\N}{\mathbb{N}}
\newcommand{\PP}{\mathbb{P}}
\newcommand{\A}{\mathcal{A}}
\renewcommand{\P}{\mathcal{P}}
\newcommand{\R}{\mathbb{R}}
\newcommand{\Z}{\mathbb{Z}}
\newcommand{\Q}{\mathbb{Q}}
\newcommand{\C}{\mathbb{C}}
\renewcommand{\a}{\alpha}
\renewcommand{\b}{\beta}
\renewcommand{\aa}{\overline{\alpha}}
\renewcommand{\epsilon}{\varepsilon}
\renewcommand{\phi}{\varphi}
\renewcommand{\rho}{\varrho}
\newcommand{\e}{\epsilon}
\newcommand{\w}{\omega}
\newcommand{\<}{\langle}
\renewcommand{\>}{\rangle}
\newcommand{\G}{\Gamma}
\renewcommand{\L}{\Lambda}
\renewcommand{\d}{\delta}
\newcommand{\GL}{\operatorname{GL}}
\newcommand{\PGL}{\operatorname{PGL}}
\newcommand{\SP}{\operatorname{Sp}}
\newcommand{\GSP}{\operatorname{GSp}}
\newcommand{\PSP}{\operatorname{PSp}}
\newcommand{\SPIN}{\operatorname{Spin}}
\newcommand{\SO}{\operatorname{SO}}
\renewcommand{\O}{\operatorname{O}}
\newcommand{\PSO}{\operatorname{PSO}}

\renewcommand{\2}{\frac12}
\newcommand{\3}{\nicefrac12}
\newcommand{\leer}{\varnothing}
\def\W#1|#2|{$\vcenter{$\w_1=#1$\hfill\lb$\w_2=#2$}$}
\def\WW#1|#2|{{$\w_1=#1$}}
\def\WWW#1|#2|{$\w\sim#1$}
\newcommand{\symb}{\hbox{\tiny$\blacksquare$}}
\newcommand{\btr}{\symb}
\newcommand{\btl}{\symb}
\newcommand{\bt}{\symb}
\newcommand{\btd}{\symb}
\newcommand{\btlu}{\symb}
\newcommand{\btld}{\symb}
\newcommand{\lb}{\hfill\linebreak}
\newcommand{\kk}{\mathfrak{k}}
\newcommand{\into}{\hookrightarrow}
\def\l/{``}%<-Damit der Spellchecker funktioniert.
\def\r/{''}
\newsavebox{\quadrat}
\savebox{\quadrat}{\linethickness{4pt}\put(0,200){\line(1,0){400}}}
\newlength{\abst}
\def\boxed#1{$\vcenter{\vbox{\vspace{\abst}\hbox{#1}\vspace{\abst}}}$}
\def\boxit#1{\boxed{#1}}

\begin{document}
\selectlanguage{english}
\title[(Quasi-)Hamiltonian manifolds of
cohomogeneity one]{(Quasi-)Hamiltonian manifolds of cohomogeneity one}
\author{Friedrich Knop} \address{Department Mathematik,
  FAU Erlangen-Nürnberg, Cauerstr. 11,
  91058 Erlangen} \email{friedrich.knop@fau.de}
\author{Kay Paulus} \address{Department Mathematik,
  FAU Erlangen-Nürnberg, Cauerstr. 11,
  91058 Erlangen} \email{math@kaypaulus.de}
\begin{abstract}
  We classify compact, connected Hamiltonian and quasi-Hamiltonian
  manifolds of cohomogeneity one (which is the same as being
  multiplicity free of rank one). The group acting is a compact
  connected Lie group (simply connected in the quasi-Hamiltonian
  case). This work is a concretization of a more general
  classification of multiplicity free manifolds in the special case of
  rank one. As a result we obtain numerous new concrete examples of
  multiplicity free quasi-Hamiltonian manifolds or, equivalently,
  Hamiltonian loop group actions.
\end{abstract}
\maketitle
\section{Introduction}
Let $K$ be a simply connected, compact Lie group. In order to study
Hamiltonian actions of the (infinite dimensional) loop group
$\mathcal LK$ on (infinite dimensional) manifolds $\mathcal M$,
Alekseev-Malkin-Meinrenken, \cite{AMM98}, introduced the notion of
quasi-Hamiltonian $K$-manifolds. These are finite dimensional
$K$-manifolds, equipped with a $2$-form
 and a momentum map.

To a certain extent, quasi-Hamiltonian manifolds are very similar to
classical Hamiltonian manifolds. This means, in particular, that a
quasi-Hamiltonian manifold can be locally described by Hamiltonian
manifolds. The most striking difference is that the target of the
momentum map is the group $K$ itself (instead of its coadjoint
representation $\kk^*$). For that reason, quasi-Hamiltonian manifolds
lack functoriality properties, like restriction to a subgroup and are
therefore more difficult to construct than Hamiltonian manifolds.

The most basic quasi-Hamiltonian manifolds are the conjugacy classes
of $K$. These are precisely the ones on which $K$ acts
transitively. The main purpose of this paper is to classify the next
case in difficulty, namely compact manifolds of cohomogeneity one,
i.e., where $K$ acts with orbits of codimension one.

Our approach is based on the papers \cite{Kno11,Kno14} where the more
general class of multiplicity free (quasi-)Hamiltonian manifolds was
considered. These are manifolds $M$ for which the momentum map induces
an injective map on orbit spaces ($M/K\into\kk^*/K$ or
$M/K\into K/\!\!/ K$, respectively). In op.\ cit. these have been
classified in terms of pairs $(\P,\G)$ where $\P$ is a convex polytope
(the momentum polytope) and $\G$ is a lattice (characterizing the
principal isotropy group). The compatibility condition between $\P$
and $\G$ is expressed in terms of a root system (affine or finite,
respectively) and the existence of certain smooth affine spherical
varieties. Even though the latter have been determined in
\cite{KVS05} (up to coverings, central tori and $\C^\times$-fibrations),
 the condition is hard to handle in practice. So the
present paper is also meant to be a (successful) test case for the
feasibility of using \cite{Kno11, Kno14} to obtain explicit
classification results. More precisely, the case considered here is
precisely that of the multiplicity free manifolds which are of rank
one, i.e., where $\rk M:=\dim\P=\rk\G=1$.
Our classification proceeds in two steps. First, there is an induction
procedure from smaller Hamiltonian manifolds. Manifolds which are not
induced are called \emph{primitive}. So, in a second step, we
determine all primitive manifolds. These can be quasi-Hamiltonian or
Hamiltonian. So, as a by-product, but of independent interest
\emph{our classification also encompasses (classical) Hamiltonian
  manifolds of cohomogeneity one}.
\begin{thm}
  Let $M$ be a primitive, multiplicity free, (quasi-)Hamiltonian
  manifold of rank one. Then $M$ corresponds to a diagram in
  \cref{ListHom}. Moreover, to each diagram there corresponds a
  manifold which is either unique (in the quasi-Hamiltonian case) or
  unique up to rescaling the symplectic structure (in the Hamiltonian
  case).
\end{thm}
In \cref{ListHom} we presented each case by a diagram which is very
close to Luna's \cite{Lun01} for classifying spherical varieties.
The classification yields many previously known quasi-Hamiltonian
manifolds. For example, the spinning $4$-sphere of Hurtubise-Jeffrey
\cite{HJ00}, Alekseev-Malkin-Woodward \cite{AMW02}) and its
generalization, the spinning $2n$-sphere, by Hurtubise-Jeffrey-Sjamaar
\cite{HJS06} are on our list. We also recover the $\SP(2n)$-action on
$\PP^n_{\mathbb{H}}$ discovered by Eshmatov \cite{Esh09} as part of a
larger series which seems to be new, namely a quasi-Hamiltonian action
of $K=\SP(2n)$ on the quaternionic Grassmannians
${\bf Gr}_k({\mathbb H}^{n+1})$ (item $(\Cc\Cc)$ for the root system
$\sC_n^{(1)}$ in \cref{ListHom}).

It should be mentioned that there is related work by Lê \cite{Le98} on
more qualitative aspects of Hamiltonian manifolds of cohomogeneity
one.

\begin{rem}
Part of this paper is based on part of the second author's doctoral
thesis \cite{Pau} which was written under the supervision of the first
named author.
\end{rem} {\bf Acknowledgment:} We would like to thank Guido Pezzini,
Bart Van Steirteghem and Wolfgang Ruppert for many explanations,
discussions and remarks. Last not least, we would like to thank the
referee for many suggestions to improve the paper.

\section{Hamiltonian and quasi-Hamiltonian manifolds}
\label{arsc}
We first recall the most important properties of Hamiltonian and
quasi-Hamiltonian manifolds.
In the entire paper, $K$ will be a compact connected Lie group with
Lie algebra $\kk$. A \emph{Hamiltonian $K$-manifold} is a triple
$(M,w,m)$ where $M$ is a $K$-manifold, $w$ is a $K$-invariant
symplectic form on $M$ and $m:M\to\kk^*$ is a smooth $K$-equivariant
map (the \emph{momentum map}) such that
\[\label{eq:mm}
  w(\xi x,\eta)=\<\xi,m_*(\eta)\> , \quad\text{for all }\xi\in
  \mathfrak{k},\ x\in M,\ \eta\in T_xM.
\]
In \cite{AMM98}, Alekseev, Malkin and Meinrenken studied this concept
in the context of loop groups. Even though these Loop groups are infinite
dimensional, the authors managed to reduce Hamiltonian loop group action to a
finite dimensional concept namely \emph{quasi-Hamiltonian
  manifolds}. These are very similar to Hamiltonian manifolds.
  
More precisely, quasi-Hamiltonian manifolds are also triples
$(M,w,m)$ where $M$ is a $K$-manifold, $w$ is a $K$-invariant
$2$-form and $m$ is a $K$-equivariant map. But there are
differences.

First of all, the Lie algebra $\kk$ has to be equipped with a
$K$-invariant scalar product. Moreover, a twist $\tau\in\Aut K$ has to
be chosen\footnote{The original paper \cite{AMM98} deals only with the
  untwisted case $\tau={\rm id}_K$. The straightforward adaption to
  the twisted case has been carried out independently in
  \cite{BoYa15}, \cite{Kno14}, and \cite{Mei15}.} (which may be the
identity). The momentum map $m$ has values in $K$ instead of $\kk^*$ and
is equivariant with respect to the $\tau$-twisted conjugation action
on $K$, i.e., $g*h:=gh\tau(g)^{-1}$. Finally, the closedness and
non-degeneracy of $w$ as well as formula \eqref{eq:mm} have to be
adapted. For the details one can consult the papers \cite{AMM98},
\cite{Kno14}, or \cite{Mei15}. They are not relevant for the present
paper.

In the following, we want to treat the Hamiltonian case and the
quasi-Hamiltonian on the same footing. So we talk about
$U$-Hamiltonian manifolds where $U=\kk^*$ in the Hamiltonian and $U=K$
in the quasi-Hamiltonian case.

This momentum map $m:M\to U$ gives rise to a map between orbits spaces:
\[
  m/K:M/K\to U/K.
\]
By definition, the fibers of this map are the {\em symplectic
  reductions} of $M$. The smooth ones are symplectic manifolds in a
natural way. In particular, they are even dimensional. Most important
for us are those manifolds for which this dimension is as low as
possible, namely $0$. These manifolds are called {\em multiplicity
  free}.
  
An important invariant of $M$ is its {\em momentum image}
$(m/K)(M/K)\subseteq U/K$. Its dimension is called the {\em rank of
  $M$}. Multiplicity free manifolds of rank zero are simply the
$K$-orbits in $U$. In this paper we study the next more difficult case
namely multiplicity free manifolds of rank one.
These two conditions can be combined into one. For this recall that
the dimension of $M/K$ is the {\em cohomogeneity of $M$}. Then we
have:
\begin{lem} For a $U$-Hamiltonian manifold $M$ the following are
  equivalent:
  \begin{enumerate}
  \item The cohomogeneity of $M$ is $1$.
  \item $M$ is multiplicity free of rank one.
  \end{enumerate}
\end{lem}
\begin{proof}
  Let $c:=\2\dim\left( (m/K)^{-1}(a)\right)$ where $a$ is a generic
  point of the momentum image. As mentioned above, it is an
  integer. By definition, $c=0$ is equivalent to multiplicity
  freeness. Let $r$ be the rank of $M$. Then we have
  \[
    \dim M/K=2c+r.
  \]
  Hence, $\dim M/K=1$ if and only if $c=0$ and $r=1$.
\end{proof}

\section{Affine root systems}
Before we go on with explaining the general structure of $U$-Hamiltonian
manifolds we need to set up notation for finite and affine root
systems. Here, we largely follow the exposition in \cite{Kno14} which
is in turn based on \cite{Mac72} and \cite{Mac03}.

Let $\overline\aaa$ be a Euclidean vector space, i.e., a finite
dimensional real vector space equipped with a scalar product
$\< \cdot,\cdot \>$. Let $\aaa$ be an affine space for
$\overline\aaa$, i.e., $\aaa$ is equipped with a free and transitive
$\overline\aaa$-action. We denote the set of affine linear functions
on $\aaa$ by $A(\aaa)$. The gradient of a function $\a\in A(\aaa)$ is
the element $\aa\in\overline{\aaa}$ with
\[\label{afflin}
  \a(X+t)=\a(X)+\< \aa, t\>,\quad X\in \aaa,\ t\in \overline{\aaa}.
\]
A reflection $s$ is an isometry of $\aaa$ whose fixed
point set is an affine hyperplane. If that hyperplane is the zero-set
of $\alpha\in A(\aaa)$ then one can express $s=s_\a$ as
$s_\a(X)=X-\a(X)\aa^\vee$ with the usual convention
$\aa^\vee=\frac{2\aa}{||\aa||^2}$.
\begin{defn} 
  An \emph{affine root system} on $\aaa$ is a subset
  $\Phi\subset A(\aaa)$ such that:
  \begin{enumerate}
  \item $\R1\cap \Phi=\leer$ (in particular $0\not\in\Phi$),
  \item $s_\a(\Phi)=\Phi$ for all $\a\in\Phi$,
  \item $\< \overline \beta, \overline \alpha ^\vee \> \in \Z$ for all
    $\a,\beta \in \Phi$,
  \item the Weyl Group $W=\< s_\a,\a\in\Phi\>$ acts
    properly discontinuously on $\aaa$,
  \item $\R\a\cap\Phi=\{+\a, -\a\}$ for all $\a\in\Phi$.
  \end{enumerate}
\end{defn}
Observe that, with our definition, $\Phi$ might be finite or even
empty. In that case, the roots have a common zero which we can use as
a base point. This way, we can identify $\aaa$ with $\overline\aaa$
and we have $\a(X)=\<\aa,X\>$ for all roots $\a$.
If $(\aaa_1,\Phi_1),\ldots,(\aaa_s,\Phi_s)$ are affine root systems
then
\[
  (\aaa_1,\Phi_1)\times\ldots\times(\aaa_s,\Phi_s)
  :=(\aaa_1\times\ldots\times\aaa_s,p_1^*\Phi_1\cup\ldots\cup p_s^*\Phi_s)
\]
is also one (where the $p_i$ are the projections). Conversely, every
affine root system admits such a decomposition such that the Weyl
group $W_i$ of $\Phi_i$ is either trivial or acts irreducibly on
$\overline\aaa_i$. We say that $\Phi$ is {\em properly affine} if each
irreducible factor $\Phi_i$ is infinite.
 
A chamber of $\Phi$ is a connected component of
$\aaa\setminus\bigcup_{\a\in\Phi}\{\a=0\}$. The closure $\A$ of a
chamber is called an {\em alcove}. If $\Phi$ is finite then $\A$ is
called a {\em Weyl chamber}. If $\Phi$ is irreducible then $\A$ is
either a simplicial cone if $\Phi$ is finite or a simplex if $\Phi$ is
properly affine.

A root $\a\in\Phi$ is called {\em simple} with respect to an alcove
$\A$ if $\A\cap\{\a=0\}$ is a wall of $\A$. The set of simple roots
(for a fixed alcove) will be denoted by $S$.

Put $\overline\Phi:=\{\aa\mid\a\in\Phi\}$ and
$\overline\Phi^\vee:=\{\aa^\vee\mid\a\in\Phi\}$. These are possibly
non-reduced finite root systems on $\overline\aaa$. We define:
\begin{defn}An {\em integral root system} on $\aaa$ is a pair
  $(\Phi, \Xi)$ where $\Phi\subset A(\aaa)$ is an affine root system
  and $\Xi\subseteq \overline\aaa$ is a lattice with
  $\overline\Phi\subseteq \Xi$ and
  $\<\Xi,\overline\Phi^\vee\>\subseteq\Z$. The integral root system is
  {\em simply connected} if
  $\Xi=\{\w\in\aaa\mid \<\w,\overline\Phi^\vee\>\subseteq\Z\}$.
\end{defn}
The classification of irreducible (infinite) affine root systems as it
can be found, e.g., in \cite{Kac90} is recalled in
\cref{table:Dynkin}. In that table, also the \emph{Dynkin label}
$k(\a)$ of each $\a\in S$ is given. These labels are uniquely
characterized by being integral, coprime, and having the property that
\[\label{eq:d}
  \d:=\sum_{\a\in S}k(\a)\a
\]
is a positive constant function.

\section{Classification of multiplicity free Hamiltonian and
  quasi-Hamiltonian manifolds}
We summarize some known facts about the quotient $U/K$.

If $U=\kk^*$ then it is classical that $U/K$ is parameterized by a
Weyl chamber for the finite root system attached to $K$.

If $U=K$ {\em we need to assume that $K$ is simply connected which we
  do from now on}.
  
Then $U/K$ is in bijection with the alcove $\A$ for a properly affine
root system which is determined by $K$ and the action of $\tau$ in the
Dynkin diagram of $K$, cf. \cite{MW04} for details. Recall that in
this case, $\kk$ is equipped with a scalar product. We use it to
identify $\kk$ with $\kk^*$. Thereby, we obtain a map
\[
  \psi:\kk^*=\kk\overset\exp\to K=U.
\]
In the Hamiltonian case, we put for compatibility reasons
$\psi={\rm id}_{\kk^*}$. Likewise, we assume that a scalar product has
been selected on $\kk$ even though the results will not depend on it.
\begin{thm} Let $K, U$ be as above. Then there is a subspace
  $\aaa\subseteq\kk^*$ and an integral root system $(\Phi,\Xi)$ on
  $\aaa$ such that:
  \begin{enumerate}
  \item If $\A\subseteq\aaa$ is any alcove of $\Phi$, then the map
    $\psi/K:\A\to U/K$ is a homeomorphism.
  \item If $X\in\A$ and $a:=\psi(X)\in U$, then the isotropy group
    \[
      K_a=\{k\in K\mid k\cdot a=a\}
    \]
    is connected, $\aaa\subseteq\kk_a$ is a Cartan subalgebra, the
    weight lattice of $K_a$ is $\Xi$, and
    \[\label{eq:SX}
      S(X):=\{\a\in S\mid \a(X)=0\},
    \]
    is a set of simple roots of $K_a$. Here $S\subset\Phi$ is the set
    of simple roots with respect to $\A$.
  \end{enumerate}
\end{thm}
Since $K_a$ depends only on $S(X)\subseteq S$ we also write
$K_a=K_{S(X)}$.
Let $M$ be a compact, connected $U$-Hamiltonian manifold. Then
the \emph{invariant momentum map} is the composition
\[
  m_+:M\overset m\to U\to U/K\overset\sim\to\A\subseteq\aaa.
\]
Its image $\P_M:=m_+(M)\subseteq\A$ can be shown to be a convex
polytope \cite{Kir84}, the so-called {\em momentum polytope} of $M$. It is the
first main invariant of $M$.

A second invariant comes from the facts that for generic $a\in\P_M$ the
isotropy group $K_a$ acts on the momentum fiber $m_{+}^{-1}(a)$ via a
quotient $A_M$ of $K_a$ which is a torus independent of $a$. Its
character group $\Gamma_M$ is a subgroup of the weight lattice $\Xi$.
\begin{thm}[\cite{Kno11},\cite{Kno14}] Let $M_1$ and $M_2$ be two
  compact, connected multiplicity free $U$-Hamiltonian manifolds with
  $\P_{M_1}=\P_{M_2}$ and $\Gamma_{M_1}=\Gamma_{M_2}$. Then $M_1$ and  $M_2$ are isomorphic as $U$-Hamiltonian manifolds.
\end{thm}

This begs the question which pairs $(\P,\Gamma)$ arise this
way. The key to the answer lies in the paper \cite{Bri87} of Brion
which connects the theory of multiplicity free Hamiltonian manifolds
with the theory of complex spherical varieties. In the following we
summarize only a simplified version which suffices for our purposes.

We start with a connected, reductive, \emph{complex} group $G$. An
irreducible algebraic $G$-variety $Z$ is called \emph{spherical} if a
Borel subgroup of $G$ has an open orbit. Now assume also that $Z$ is
affine and let $\C[Z]$ be its ring of regular functions. Then the
Vinberg-Kimelfeld criterion \cite{VK78} asserts that $Z$ is spherical
if and only if $\C[Z]$ is multiplicity free as a $G$-module. This
means that there is a set (actually a monoid) $\Lambda_Z$ of dominant
integral weights of $G$ such that
\[\label{eq:CZ}
  \C[Z]\cong\bigoplus_{\chi\in\Lambda_Z}V_\chi
\]
where $V_\chi$ is the simple $G$-module of highest weight $\chi$. A
theorem of Losev \cite{Los06} asserts that in case $Z$ is smooth, the
variety $Z$ is in fact uniquely determined by its weight monoid
$\Lambda_Z$.

Let $K\subseteq G$ be a maximal compact subgroup. Then any smooth
affine $G$-variety can be equipped with the structure of a Hamiltonian
$K$-manifold by embedding $Z$ into a finite dimensional $L$-module $V$
and using a $K$-invariant Hermitian scalar product on $V$ to define a
momentum map. Then

\begin{enumerate}
\item $Z$ is spherical as a $G$-variety if and only if it is
  multiplicity free as a Hamiltonian $K$-manifold.
\item $\P_Z=\R_{\ge0}\Lambda_Z$ (the convex cone generated by $\Lambda_Z$).
\item $\Gamma_Z=\Z\Lambda_Z$ (the group generated by $\Lambda_Z$).
\end{enumerate}
The first two items were proved by Brion \cite{Bri87} in the context
of projective varieties. The version which we need, namely for affine
varieties, was proved by Sjamaar in \cite{Sja98}. For the last item
see Losev \cite[Prop.~8.6(3)]{Los06}.
\begin{rem}
It follows from the normality of $Z$ that conversely
\[\label{eq:cone}
\Lambda_Z=\P_Z\cap\G_Z.
\]
So $\Lambda_Z$ and the pair $(\P_Z,\G_Z)$ carry the same information.
\end{rem}
\begin{defn}
  A pair $(\P,\Gamma)$ is called \emph{$G$-spherical} if there exists a
  smooth affine spherical $G$-variety $Z$ with
  $\P=\R_{\ge0}\Lambda_Z$ and $\Gamma=\Z\Lambda_Z$. The (unique) smooth
  variety $Z$ will be called a \emph{model for $(\P,\G)$}.
\end{defn}
Now we go back to $U$-Hamiltonian manifolds. For any subset
$\P\subseteq\A$ and point $X\in\P$ we define the {\em tangent cone} of
$\P$ at $X$ as
\[
  T_X\P:=\R_{\ge0}(\P-X).
\]
Here is a local version of sphericality:
\begin{defn}
  Let $\P\subseteq\A$ be a compact convex polytope and
  $\Gamma\subseteq\Xi$ a subgroup.
  \begin{enumerate}
  \item $(\P,\Gamma)$ is \emph{spherical in $X\in\P$} if
    $(T_X\P,\Gamma)$ is $L$-spherical where $L:=K_a^\C$ is the Levi subgroup corresponding to
    $a:=\psi(X)\in U$. The model variety for $(T_X\P,\Gamma)$ will be
    called the \emph{local model of $(\P,\G)$ in $X$}.
  \item The pair $(\P,\G)$ is {\em locally spherical} if it is
    spherical in every vertex of $\P$.
  \end{enumerate}
\end{defn}
\begin{rem}\label{rem:parallel}
  It follows from the definition of $G$-sphericality that in
  a locally spherical pair $\P$ and $\G$ are necessarily parallel
  in the sense that $\P$ is a polytope of maximal dimension
  inside the affine subspace $X+\<\G\>_\R\subseteq\aaa$ for any
  $X\in\P$.
\end{rem}
The classification theorem can now be stated as follows:
\begin{thm}[\cite{Kno11},\cite{Kno14}]\label{thm:Delzant}
  Let $K$ be a connected compact Lie group which is assumed to be
  simply connected in the quasi-Hamiltonian case.  Then the map
  $M\mapsto(\P_M,\G_M)$ induces a bijection between
  \begin{enumerate}
  \item isomorphism classes of compact, connected multiplicity free
    $U$-Hamiltonian manifolds and
  \item locally spherical pairs $(\P,\Gamma)$ where $\P\subseteq\A$ is
    a compact convex polyhedron and $\Gamma\subseteq\Xi$ is a subgroup.
  \end{enumerate}
\end{thm}
\begin{rems}
  \begin{enumerate}
  \item  The relation between pairs and manifolds can be made more
  precise. Let $M$ be a $U$-Hamiltonian manifold and $X\in \P_M$. Then
  there exists a neighborhood $\P_0$ of $X$ in $\P$ such that
\[
M_0\cong K\times^{K_a}Z_0
\]
where $M_0=m_+^{-1}(\P_0)$, $a=\psi(X)$, and $Z_0\subseteq Z$ is a
$K_a$-stable open subset of the local model $Z$ in $X$.

\item The construction of locally spherical pairs is quite
difficult. Already deciding whether a given pair is locally spherical
is intricate. There is an algorithm due to Pezzini-Van Steirteghem
\cite{PVS19} for this but we are not going to use it since in our
setting it is not necessary.
\end{enumerate}
\end{rems}

Because of this theorem, we are going to work from now on exclusively
on the \l/combinatorial side\r/, i.e., with locally spherical
pairs. We start with two reduction steps.

\begin{defn}
Let
\[
  S(X):=\{\a\in S\mid\a(X)=0\text{ for $X\in\P$}\}
\]
be the set of simple roots which are zero for a fixed $X\in \P$ and
\[
  S(\P):=\bigcup_{X\in\P}S(X)=\{\a\in S\mid\a(X)=0\text{ for some $X\in\P$}\}.
\]
\end{defn}

Thus, elements of $S(\P)$ correspond to walls of $\A$ which contain a point
of $\P$. Let $K_0:=K_{S(\P)}$ be the corresponding (twisted) Levi
subgroup of $K$. Then it is immediate that $(\P,\G)$ is locally
spherical for $K$ if and only if it is so for $K_0$. This observation
reduces classifications largely to pairs with $S(\P)=S$.

\begin{defn}
	A polyhedron $\P \subseteq \A$ is called genuine if $S(\P)=S$.
\end{defn}

There is another reduction. Assume $S_0\subseteq S$
is a component of the Dynkin diagram of $S$. It corresponds to a
(locally) direct semisimple factor $L_{S_0}$ of $G=K^\C$. Suppose also
that $S_0\subseteq S(X)$ for all $X\in\P$. Then it follows from
\cref{rem:parallel} that $\<\G,S_0\>=0$. In turn \eqref{eq:CZ} implies
that $L_{S_0}$ will act trivially on every local model $Z$ of
$(\P,\G)$. This means that also the roots in $S_0$ can be ignored
for determining the sphericality of $(\P,\G)$.

\begin{defn}
  A genuine polyhedron $\P\subseteq\A$ is called \emph{primitive} if
  $S$ does not contain a component $S_0$ with $S_0\subseteq S(X)$ for
  all $X\in\P$.
\end{defn}

The following lemma summarizes our findings
  \begin{lem}\label{lemma:reduction}
    Let $\P\subseteq\A$ be a compact convex polyhedron and let be
    $\G\subseteq\Xi$ a subgroup. Let
\[
S^c:=\{\a\in S\mid \a(X)\ne0\text{ for all }X\in\P\}
\]
and let $S_1$ be the union of all components $C$ of $S\setminus S^c$
with $C\subseteq S(X)$ for all $X\in\P$. Let
$\overline\Xi:=\Xi\cap S_1^\perp$. Then $\P$ is primitive for
$\overline S:=S\setminus (S^c\cup S_1)$. Moreover, the pair $(\P,\G)$
is locally spherical for $(S,\Xi)$ if and only if it is so for
$(\overline S,\overline\Xi)$.
\end{lem}

Now the purpose of this paper is to present a complete classification
of primitive locally spherical pairs in the special case when
$\rk M=\dim\P=1$.

In this case the following simplifications occur: the polyhedron $\P$
is a line segment $\P=[X_1,X_2]$ with $X_1, X_2\in\A$ and $\G=\Z\w$
with $\w\in\Xi$. It follows from \cref{rem:parallel} that
  \[\label{eq:XcX}
    X_2=X_1+c\w\text{ for some }c\ne0.
  \]
  By replacing $\w$ by $-\w$ if necessary, we may assume that
  $c>0$. Then \cref{thm:Delzant} boils down to:
\begin{cor}\label{cor:Kclass}
  The map $M\mapsto (\P_M,\G_M)=([X_1,X_2],\Z\w)$ induces a bijection
  between
  
  \begin{enumerate}
  \item isomorphism classes of compact, connected multiplicity free
    $U$-Hamiltonian manifolds of rank one and
  \item\label{it:Kclass2} triples $(X_1, X_2, \w)$ satisfying
    \eqref{eq:XcX} such that $\N\w$ is the weight monoid of a smooth
    affine spherical $K_{S(X_1)}^\C$-variety $Z_1$ and $\N(-\w)$ is
    the weight monoid of a smooth affine spherical
    $K_{S(X_2)}^\C$-variety $Z_2$. The triples $(X_1, X_2, \w)$ and
    $(X_2, X_1, -\w)$ are considered equal.
  \end{enumerate}
\end{cor}
Triples as above will be called {\em spherical}. A triple is
\emph{genuine} or \emph{primitive} if $\P=[X_1,X_2]$ has this
property. The varieties $Z_i$ are called the {\em local models} of the
triple.
\section{The local models}
We proceed by recalling all possible local models, i.e., smooth,
affine, spherical $L$-varieties $Z$ of rank one where $L$ is a connected,reductive, complex, algebraic group. Then
\[
  \C[Z]=\bigoplus_{n\in\L}V_{n\w},
\]
where $\w$ is a non-zero integral dominant weight, $V_{n\w}$ is the
simple $L$-module of highest weight $n\w$, and $\L$ equals either $\N$ or$\Z$.
The case $\L=\Z$ is actually irrelevant for our purposes since this
case only occurs as local model of an interior point of $\P$ (by
\eqref{eq:cone}).

In case $\L=\N$, the weight $\w$ is unique.
\begin{thm}\label{thm:SASV1}
  Let $Z$ be a smooth, affine, spherical $L$-variety of rank one. Then
  one of the following cases holds:
  \begin{enumerate}
  \item\label{it:SASV1-1} $Z=\C^*$ and $L$ acts via a
    non-trivial character.
  \item\label{it:SASV1-2} $Z=L_0/H_0$ where $(L_0,H_0)$ appears in
    the first part of \cref{table:local} and $L$ acts via a surjective
    homomorphism $\phi:L\to L_0$.
  \item\label{it:SASV1-3} $Z=V_0$ where $(L_0,V_0)$ appears in the
    second part of \cref{table:local} and $L$ acts via a homomorphism
    $\phi:L\to L_0$ which is surjective modulo scalars (except for
    case $\Ca_0$ when $\phi$ should be surjective).
  \end{enumerate}
\end{thm}
\begin{proof}
  Smooth affine spherical varieties have been classified by Knop-Van
  Steirteghem in \cite{KVS05} and the assertion could be extracted
  from that paper. A much simpler argument goes as follows. First,
  a simple application of Luna's slice theorem (see
  \cite[Cor.~2.2]{KVS05}) yields $Z\cong L\times^HV$ where
  $H\subseteq L$ is a reductive subgroup and $V$ is a representation
  of $H$. As the homogeneous space $L/H$ is the image of
  $Z=L\times^H V$ under the projection $Z\to L/H$. The rank of the
  homogeneous space $L/H$ is at most the rank of $Z$, so either $0$ or
  $1$.
  
  If it is $0$, then $L/H$ is projective (see, e.g.,
  \cite[prop. 10.1]{Tim11}), but, being also affine, it is a single
  point, i.e., $L=H$. We deduce $Z=V$, i.e., $Z$ is a spherical module
  of rank one. The classification of spherical modules
  (Kac~\cite{Kac80}, see also \cite{Kno98}), yields the cases in
  \ref{it:SASV1-3}.
  
  Assume now that $L/H$ has rank one. This means that $Z$ and $L/H$
  have the same rank. Let $F\subseteq L$ be the stabilizer of a point
  in the open $L$-orbit of $Z$ such that $F\subseteq H$. By
  \cite[lem. 2.4]{Gan10}, the quotient $H/F$ is finite. This implies
  that the projection $Z\to H/F$ has finite fibers. Hence $V=0$ and
  $Z=L/H$ is homogeneous. The classification of homogeneous spherical
  varieties of rank one (Akhiezer~\cite{Akh83}, see also \cite{Bri89},
  and Wasserman~\cite{Was96}), yields the cases in \ref{it:SASV1-1}
  and \ref{it:SASV1-2}. Observe, that the non-affine cases of
  Akhiezer's list have been left out.
\end{proof}
\begin{rems} Some remarks concerning \cref{table:local}:
  \begin{enumerate}
    
  \item Observe that items $[\2]d_2$ and $[\2]d_3$ could be made part of
  the series $[\2]d_n$. Because of their singular behavior we chose
  not to do so. For example both can be embedded into
  $\sA_n$-diagrams. Moreover, $[\2]d_2$ are the only cases with a
  disconnected Dynkin diagram.
  
\item We encode the local models by the diagram given in the last
  column of \cref{table:local}. For homogeneous models these diagrams
  are due to Luna \cite{Lun01}. The inhomogeneous ones are specific
  to our situation.
  
\item For a homogeneous model the weight $\w$ is a fixed linear
  combination of simple roots (recorded in the fourth column). Hence
  it lifts uniquely to a weight of $L$. On the other hand, for
  inhomogeneous models the weight of $L$ is only unique up to a
  character. This is indicated by the notation $\w\sim\pi_1$ which
  means that $\<\w,\aa^\vee\>=1$ for $\a=\a_1$ and $=0$ otherwise.
\end{enumerate}
\end{rems}
Let $S$ be the set of simple roots of $L_0$, i.e., the set of vertices
of a diagram. Then inspection of \cref{table:local} shows that the
elements of
  \[\label{eq:Sprime}
    S':=\{\a\in S\mid \<\w,\aa^\vee\>>0\}
  \]
  are exactly those which are decorated. All other simple roots $\a$
  satisfy $\<\w,\aa^\vee\>=0$. Another inspection shows that the
  diagram of a local model is almost uniquely encoded by the pair
  $(S,S')$. What is getting lost is a factor $c$ of $\3$, $1$ or
  $2$, and the cases $a_1$ and $\Ca_1$ become
  indistinguishable. So we assign the formal symbol $c=i$ to the
  inhomogeneous cases. This way, the local model is uniquely encoded
  by the triple $(S,S',c)$ with $c\in\{\3,1,2,i\}$ which triggers the
  following
\begin{defn}
  A \emph{local diagram} is a triple $\cD=(S,S',c)$ associated to a
  local model in \cref{table:local}. In the homogeneous case, let
  $\w_\cD$ be the weight given in column 4. If $\cD$ is inhomogeneous
  and $S$ is non-empty then $\a_\cD$ denotes the unique element of
  $S'$. Moreover, we put $\aa^\vee_\cD:=\overline{\a_\cD}^\vee$.
  \end{defn}
  \section{The classification}
  Let $(X_1,X_2,\w)$ be a primitive spherical triple. Then we obtain
  two local models $Z_1,Z_2$ which determine two local diagrams
  $\cD_1=(S_1,S_1',c_1)$, $\cD_2=(S_2,S_2',c_2)$ where
  $S_1,S_2\subseteq S$. Put
  \[
    S^p(\w):= \{\a\in S\mid\<\w,\aa^\vee\>=0\}.
  \]
\begin{lem}\label{lem:cap}
  Let $(X_1,X_2,\w)$ be a primitive spherical triple. Then
  \[
    S(X_1)\cup S(X_2)=S\text{ and }S(X_1)\cap S(X_2)=S^p(\w).
  \]
\end{lem}
\begin{proof}
  The first equality holds because the triple is genuine. The
  inclusion $S(X_1)\cap S(X_2)\subseteq S^p(\w)$ follows directly from
  \eqref{eq:XcX}. Assume conversely that $\a\in S^p(\w)$. Without loss
  of generality we may assume that also $\a\in S(X_1)$. But then also
  $\a\in S(X_2)$ by \eqref{eq:XcX}.
\end{proof}
From now let $i\in\{1,2\}$ and $j:=3-i$, so that if $Z_i$ is a local
model then $Z_j$ is the other.
\begin{lem}
  Let $(X_1,X_2,\w)$ be a primitive spherical triple and $\cD_1,\cD_2$
  as above. Then
  \[\label{eq:SSpS}
    S=S_1'\dot\cup S^p(\w)\dot\cup S_2'.
  \]
  Moreover,
  \[\label{eq:SSSS}
    S(X_i)=S_i'\dot\cup S^p(\w)=S\setminus S_j'.
  \]
\end{lem}
\begin{proof}
  It follows from \cref{thm:SASV1} that every $\a\in S(X_i)$ (a simple
  root of $L$) is either in $S_i$ (a simple root of $L_0$) or a
  simple root of $\ker\phi$. In the latter case, we have
  $\<\w,\aa^\vee\>=0$.
  
  Now let $\a\in S$ and assume first $\<\w,\aa^\vee\>>0$. Since the
  triple is genuine we have $S=S(X_1)\cup S(X_2)$. If $\a\in S(X_2)$
  then actually $\a\in S_2$. This contradicts $\<-\w,\aa^\vee\>\ge0$ for all
  $\a\in S_2$. Thus $\a\in S(X_1)$. By the same reasoning we have
  $\a\in S_1$. But then $\a\in S_1'$ by \eqref{eq:Sprime}.
  
  Analogously, $\<\w,\aa^\vee\><0$ implies $\a\in S_2'$. This proves
  \eqref{eq:SSpS}. The second equality \eqref{eq:SSSS} now follows
  from \cref{lem:cap}.
\end{proof}
\begin{defn}
  Let $S'$ be a subset of a graph $S$. The \emph{connected closure}
  $C(S',S)$ of $S'$ in $S$ is the union of all connected components of
  $S$ which meet $S'$. In other words, $C(S',S)$ is the set of
  vertices of $S$ for which there exists a path to $S'$.
\end{defn}
The following lemma shows in particular how to recover $S_i$ from the
triple $(S,S_1',S_2')$.
\begin{lem}\label{lem:prim}
  Let $(X_1,X_2,\w)$ be primitive. Then
  \begin{enumerate}
  \item\label{it:prim1} $S_i$ is the connected closure of $S_i'$ in
    $S\setminus S_j'$.
  \item\label{it:prim2} $S$ is the connected closure of
    $S_1'\cup S_2'$.
  \item\label{it:prim3} $S=S_1\cup S_2$.
  \end{enumerate}
\end{lem}
\begin{proof}
  \ref{it:prim1} Recall that $S\setminus S_j'=S(X_i)$ is the
  disconnected union of $S_i$ and the Dynkin diagram $C_i$ of
  $\ker\phi$. Inspection of \cref{table:local} shows that $S_i$ is the
  connected closure of $S_i'$.
  
  \ref{it:prim2} Let $C\subseteq S$ be a component with
  $C\cap(S_1'\cup S_2')=\leer$. Then
  $C\subseteq S^p(\w)=S(X_1)\cap S(X_2)$ in contradiction to primitivity.
  
  \ref{it:prim3} By \ref{it:prim1}, the connected closure of
  $S_1'\cup S_2'$ in $S$ is $S_1\cup S_2$.
\end{proof}
\begin{defn}\label{def:prim5}
  Let $\cD$ be the quintuple $\cD=(S,S_1',c_1,S_2',c_2)$ where $S$ is a
  (possibly empty) Dynkin diagram, $S_1'$, $S_2'$ are disjoint
  (possibly empty) subsets of $S$ and $c_1,c_2\in\{\2,1,2\}$.
  
   Let  $S_i$ be the connected closure of $S_i'$ in $S\setminus S_j'$.  Then
  $\cD$ is a \emph{primitive spherical diagram} if it has following
  properties:
  \begin{enumerate}
  \item\label{it:prim5-1} $S=S_1\cup S_2$.
  \item\label{it:prim5-2} The triples $\cD_i:=(S_i,S_i',c_i)$ are
    local diagrams.
  \item\label{it:prim5-3}
    \begin{enumerate}
    \item\label{it:prim5-31} If both $\cD_i$ are homogeneous then
      $\w_{\cD_1}+\w_{\cD_2}=0$.
    \item\label{it:prim5-32} If $\cD_i$ is homogeneous and $\cD_j$ is
      inhomogeneous with $S_j\ne\leer$ then
      $\<\w_{\cD_i},\aa^\vee_{\cD_j}\>=-1$.
    \item\label{it:prim5-33} If both $\cD_i$ are inhomogeneous with
      both $S_i\ne\leer$ and $S$ is affine and irreducible then
      $k(\a_{\cD_1}^\vee)=k(\a_{\cD_2}^\vee)$ where $k(\a^\vee)$ is
      the colabel of $\a$, i.e., the label of $\a^\vee$ in the dual
      diagram of $S$.
    \end{enumerate}
  \end{enumerate}
\end{defn}
A primitive spherical diagram can be represented by the Dynkin diagram
of $S$ with decorations as in \cref{table:local} which indicate the
subsets $S_i'$ and the numbers $c_i$.
 
\begin{example}
  Consider the following diagram on $\sF_4^{(1)}$:
  \[
\begin{picture}(7200,1100)(0,-300)
  \put(1200,600){\tiny\3}
  \put(0,0){\usebox{\dynkinatwo}}
  \put(1800,0){\usebox{\dynkinffour}}
  \put(1800,0){\circle*{600}}
  \put(5400,0){\circle*{600}}
    \end{picture}
  \]
  It represents the quintuple with $S_1'=\{\a_1\}$, $c_1=\3$,
  $S_2'=\{\a_3\}$, $c_2=1$. Hence $S_1=\{\a_0,\a_1,\a_2\}$, and
  $S_2=\{\a_2,\a_3,\a_4\}$. Comparing with \cref{table:local} we see
  that the local diagrams are of type $\frac12d_3$ and $c_3$,
  respectively. The diagram is bihomogeneous so we need to check
  condition \ref{it:prim5-3}\ref{it:prim5-31}. Indeed
  \[
    \w_{\cD_1}+\w_{\cD_2}=(\2\aa_0+\aa_1+\2\aa_2)+(\aa_2+2\aa_3+\aa_4)=    \2(\aa_0+2\aa_1+3\aa_2+4\aa_3+2\aa_4)=0
  \]
  (compare with the labels of $\sF_4^{(1)}$ in \cref{table:Dynkin}). Thus,
  the above diagram is primitive spherical.
\end{example}
The point of \cref{def:prim5} is of course:
\begin{cor}
  Let $(X_1,X_2,\w)$ be a primitive spherical triple with local
  diagrams $(S_1,S_1',c_1)$ and $(S_2,S_2',c_2)$. Then
  $(S,S'_1,c_1,S'_2,c_2)$ is a primitive spherical diagram.
\end{cor}
\begin{proof}
  All conditions have been verified except for \ref{it:prim5-3}\
  \ref{it:prim5-33}. If $S$ is affine then the coroots satisfy the
  linear dependence relation
  \[\label{eq:kvee}
    \sum_{\a\in S}k(\a^\vee)\aa^\vee=0.
  \]
  We pair this with $\w$ and observe that $\<\w,\aa^\vee\>
  =1,-1,0$
  according to $\a=\a_{\cD_1}$, $\a=\a_{\cD_2}$ or otherwise. This
  implies the claim.
\end{proof}
The following is our main result. It will be proved in the next
section.
\begin{thm}\label{thm:main}
  Every primitive spherical diagram is isomorphic to an entry of
  \cref{ListHom}.
\end{thm}
Explanation of \cref{ListHom}: The first column gives the type of
$S$. The second lists for identification purposes the type of the
local models. The diagram is given in the fifth column. If a parameter
is involved, its scope is given in the last column. Observe the
boundary cases where we used the conventions $b_1=a_1$, $2b_1=2a_1$,
$c_2=b_2$, and $\Cc_1=\Ca_1$. In some cases, besides
$(S,S_1',c_1,S_2',c_2)$ also $(S,S_1',c\,c_1,S_2',c\,c_2)$ is
primitive spherical where $c$ is the factor in the column
\l/factor\r/. An entry of the form $[c]_{n=a}$ indicates that the factor
applies only to the case $n=a$.

Finally, the weight $\w$ can be read off from the third column. More
precisely, if $\cD_i$ is homogeneous then $\w_i$ indicates the unique
lift of $\w$ or $-\w$ to an affine linear function with
$\w_i(X_i)=0$. If both local models $\cD_1,\cD_2$ are inhomogeneous
then $\w$ is only unique up to a character of $G$. Thus, the notation
$\w\sim\w_0$ means $\<\w,\a^\vee\>=\<\w_0,\a^\vee\>$ for all
$\a\in S$. Here, $\pi_i\in\Xi\otimes\Q$ denotes the $i$-th fundamental
weight.

\begin{example}\label{ex:a11a22}
  The primitive diagrams for $\sA_1^{(1)}$ and $\sA_2^{(2)}$ are
  \[\label{eq:a11a22}
    \begin{picture}(2400,1800)(-300,-900)
      \put(-400,-300){$\bt$} \put(0,0){\usebox\leftrightbiedge}
      \put(1500,-300){$\bt$}
    \end{picture}\qquad
    \begin{picture}(2400,1800)(-300,-900) \put(0,0){\usebox\aone}
      \put(0,0){\usebox\leftrightbiedge} \put(1800,0){\usebox\aone}
    \end{picture}\qquad
    \begin{picture}(2400,1800)(-300,-900) \put(0,0){\usebox{\aprime}}
      \put(0,0){\usebox\leftrightbiedge}
      \put(1800,0){\usebox{\aprime}}
    \end{picture}\qquad
    \begin{picture}(2400,1800)(-300,-900)
      \multiput(0,0)(1800,0){2}{\circle*{300}} %\thicklines
      \multiput(0,-67)(0,133){2}{\line(1,0){1800}}
      \multiput(0,-200)(0,400){2}{\line(1,0){1800}}
      \multiput(150,0)(25,25){20}{\circle*{50}}
      \multiput(150,0)(25,-25){20}{\circle*{50}}
      \put(0,0){\usebox\aone} \put(1500,-300){$\bt$}
    \end{picture}\qquad
    \begin{picture}(2400,1800)(-300,-900)
      \multiput(0,0)(1800,0){2}{\circle*{300}} %\thicklines
      \multiput(0,-67)(0,133){2}{\line(1,0){1800}}
      \multiput(0,-200)(0,400){2}{\line(1,0){1800}}
      \multiput(150,0)(25,25){20}{\circle*{50}}
      \multiput(150,0)(25,-25){20}{\circle*{50}}
      \put(0,0){\usebox\aprime} \put(1800,0){\usebox\aone}
    \end{picture}
  \]
  In these cases, we have $\P=\A$ and
  $\w=\frac12\aa_1,\aa_1,2\aa_1,\frac12\aa_1,\aa_1$, respectively
  (where $S=\{\a_0,\a_1\}$).
\end{example}

The conditions defining a primitive spherical diagram $\cD$ have been
shown to be necessary but it is not clear whether each of them can be
realized by a (quasi-)Hamiltonian manifold $M$. And if so, how unique
is $M$? We answer these questions in \cref{thm:real} below. To state
it we need more notation.

\begin{defn}
  \begin{enumerate}
  \item For a finite root system $\Phi$ let $\pi_\a$ be the
    fundamental weight corresponding to $\a\in S$.
  \item If $\Phi$ is affine and irreducible (hence $\A$ is a simplex)
    let $P_\a\in\A$ be the vertex of $\A$ with $\a(P_\a)>0$.
  \end{enumerate}
\end{defn}
Let $\cD$ be a local diagram $\ne\Ca_0$. An inspection of
\cref{table:local} shows that the pairing $\<\w,\aa^\vee\>$ is in fact
independent of $\a\in S'$ (actually only $\a_{n\ge2}$ and
$[\2]d_{n\ge2}$ have to be checked). The common value will be denoted
by $n_\cD$. Here is a list:
  
\begin{lis}{The numbers $n_\cD$}\label{eq:nD}
  \setlength{\tabcolsep}{3pt}
    \begin{longtable}{>{$}l<{$}|>{$}c<{$}>{$}c<{$}>{$}c<{$}>{$}c<{$}>{$}c<{$}>{$}c<{$}>{$}c<{$}>{$}c<{$}>{$}c<{$}>{$}c<{$}>{$}c<{$}>{$}c<{$}>{$}c<{$}>{$}c<{$}>{$}c<{$\hspace{-4pt}}}
      \cD\vrule width 0pt depth 10pt&a_1&2a_1&a_{n\ge2}&b_{n\ge2}&2b_{n\ge2}&c_{n\ge3}&d_{n\ge2}&\2d_{n\ge2}&f_4&g_2&2g_2&b_3'&\2b_3'&\Ca_{n\ge1}&\Cc_{n\ge2}\\
      \hline
      n_\cD\vrule width 0pt height 15pt&2&4&1&1&2&1&2&1&1&1&2&2&1&1&1\\
    \end{longtable}
  \end{lis}
\begin{thm}\label{thm:real}
  Let $K$ be simply connected (also in the Hamiltonian case) and let
  $\cD\ne(\leer)$ be a primitive diagram for $(\aaa,\Phi,\Xi)$.
  \begin{enumerate}
  \item\label{eq:Xfin} If $\Phi$ is finite then $\cD$ can be realized by a
    multiplicity free Hamiltonian manifold of rank one. This manifold
    is unique up to a positive factor of the symplectic structure. The
    momentum polytope is given by
    \[
      X_i=c\,n_{\cD_j}\sum_{\a\in S_j'}\pi_\a
    \]
    (see \cref{eq:nD} for $n_\cD$) if both $S_i'$ are non-empty. If
    $S_1'\ne\leer$ and $S_2'=\leer$ then
    \[
      X_1=0\text{ and } X_2=c\w.
    \]
    In both cases, $c$ is some arbitrary positive factor.
  \item If $\Phi$ is infinite and irreducible then $\cD$ can be
    realized by a unique multiplicity free quasi-Hamiltonian manifold
    of rank one. The momentum polytope is given by
    \[\label{eq:Xaffin}
      X_i=\begin{cases}
        P_\a&\text{if }S_j'=\{\a\}\\
        \frac{k(\a^\vee)}{k(\a^\vee)+k(\b^\vee)}P_\a+
      \frac{k(\b^\vee)}{k(\a^\vee)+k(\b^\vee)}P_\b&
      \text{if }S_j'=\{\a,\b\}.
    \end{cases}
  \]
  
  \item If $\Phi$ is infinite and reducible (cases
    $\sA_1^{(1)}\times \sA_1^{(1)}$ and $\sA_2^{(2)}\times \sA_2^{(2)}$) then
    $\cD$ can be realized by a multiplicity free quasi-Hamiltonian
    manifold of rank one if and only if the scalar product is chosen
    to be the same on both factors of $K$, i.e., if the alcove $\A$ is
    a metric square. This manifold is then unique.
  \end{enumerate}
\end{thm}
\begin{proof}
  Let $L_i\subseteq K^\C$ be the (twisted) Levi subgroup having the
  simple roots $S(i):=S\setminus S_j'$. We have to construct
  $(X_1,X_2,\w)$ such that $S(X_i)=S(i)$, $X_2-X_1\in\R_{>0}\w$, and
  $\w$ or $-\w$ generates the weight monoid of a smooth affine spherical
  $L_1$-variety or $L_2$-variety, respectively.
  
  If both $\cD_1$ and $\cD_2$ are homogeneous then there are exactly two
  choices for $\w$ namely $\w_{\cD_1}$ and $\w_{\cD_2}$ which are
  related by $\w_{\cD_1}+\w_{\cD_2}=0$. We claim that
  $\<\w_{\cD_i},\aa^\vee\>\in\Z$ for all $\a\in S$. This follows by
  inspection for $\a\in S_i$. From $\w_{\cD_1}=-\w_{\cD_2}$ we get it
  also for $\a\in S_j$. Since $K$ is simply-connected, the weights
  $\w_{\cD_i}$ are integral, i.e., $\w\in\Xi$.
  
  If $\cD_1$ is homogeneous and $\cD_2$ is inhomogeneous we must put
  $\w=\w_{\cD_1}$. By condition \ref{it:prim5-2}\ref{it:prim5-32} of
  \cref{def:prim5} we have $\<-\w,\aa^\vee_{\cD_2}\>=1$. Let
  $\b\in S_2\setminus\{\a_{\cD_2}\}$. Then $\b$ is not connected to
  any $\a\in S_1$ by \cref{lem:prim}\ref{it:prim1}. From
  $\w\in\Q S_1$ we get $\<\w,\overline\b^\vee\>=0$. Hence $\w\in\Xi$
  and both $\N\w$ and $\N(-\w)$ form the weight monoid of a smooth
  affine spherical $L_1$- or $L_2$-variety, respectively.
  
  If both $\cD_i$ are inhomogeneous then we need a weight $\w$ with
  $\<\w,\aa^\vee_{\cD_1}\>=1$, $\<\w,\aa^\vee_{\cD_2}\>=-1$, and
  $\<\w,\aa^\vee\>=0$ otherwise. If $\Phi$ is finite then $\w$ exists
  and is unique since $S$ is a basis of $\Xi\otimes\Q$. If $S$ is
  affine and irreducible then $\w$ exists and is unique because of
  condition \ref{it:prim5-3}\ref{it:prim5-33} of \cref{def:prim5}. In
  both cases $\w$ is integral. The case of reducible affine root
  systems will be discussed at the end.
  
  This settles the reconstruction of $\w$. It remains to construct
  points $X_1,X_2\in\A$ with $S(X_1)=S(1)$, $S(X_2)=S(2)$ and
  $X_2-X_1\in\R_{>0}\w$. These boil down to the following set of
  linear (in-)equalities (where the last column just records the known
  behavior of $\w$):
  
  \[\label{eq:X1}
    \begin{array}{l|ll||l}
      \a&X_1&X_2&\w\\
      \hline
      \a\in S_1'&\a(X_1)=0&\a(X_2)>0&\<\w,\aa\>>0\\
      \a\in S_2'&\a(X_1)>0&\a(X_2)=0&\<\w,\aa\><0\\
      \a\not\in S_1'\cup S_2'&\a(X_1)=0&\a(X_2)=0&\<\w,\aa\>=0\\
    \end{array}
  \]
  \[\label{eq:c}
    X_2=X_1+c\w\text{ with }c>0.
  \]
  
  The inequalities \eqref{eq:X1} for $X_1$ define the relative
  interior of a face of the alcove $\A$ (observe that
  $S_2'\ne\leer$ if $\Phi$ is affine). The first and the third set of
  inequalities for $X_2$ then follow from \eqref{eq:c}. Inserting
  \eqref{eq:c} into the second set we get equalities for $X_1$ and
  $c$:
  \[\label{eq:cc}
    \a(X_1)=c\<-\w,\aa\>>0\text{ for all }\a\in S_2'.
  \]
  Define the affine linear function $\a^\vee:=\frac2{\|\aa\|}\a$.
  Then \eqref{eq:cc} is equivalent to
  \[\label{eq:aVX}
    \a^\vee(X_1)=cn_{\cD_2}>0\text{ for all }\a\in S_2'.
  \]
  This already shows assertion \ref{eq:Xfin} of the theorem.
  Now assume that $\Phi$ is affine and irreducible. Then there is the
  additional relation
  \[\label{eq:kae}
    \sum_{\b\in S}k(\b^\vee)\b^\vee(X)=\e\equiv\text{const.}>0,\ X\in\aaa.
  \]
  Setting $X=X_1$, we get
  \[
    cn_{\cD_2}\sum_{\b\in S_2'}k(\b^\vee)=\e.
    \]
    This means that $c$ is unique and positive. From \eqref{eq:aVX} we
    get
    \[
\a^\vee(X_1)=[\sum_{\b\in S_2'}k(\b^\vee)]^{-1}\e.
      \]
      Evaluation of \eqref{eq:kae} at $X=P_\a$ yields
      $\a^\vee(P_\a)=\frac{\e}{k(\a^\vee)}$. To obtain
      \eqref{eq:Xaffin} just observe that $S_2'$ has either one or two
      elements.
      
      Finally, assume $\Phi$ is reducible. The mixed types \l/finite
      times infinite\r/ do not appear in our context. For the two other
      cases, the existence of $([X_1,X_2],\w)$ is clear from the
      following graphics. In particular, they show why $\A$ must be a
      metric square.
      
      \definecolor{uuuuuu}{rgb}{0.26666666666666666,0.26666666666666666,0.26666666666666666}
      
      \begin{center}
        \begin{tabular}{ccc}
$A_1^{(1)} \times A_1^{(1)}$ &$A_1^{(1)} \times A_1^{(1)}$ & $A_2^{(2)} \times A_2^{(2)}$
\\
\begin{tikzpicture}[x=0.25cm,y=0.25cm,>=angle 60]
\fill[line width=2.pt,color=uuuuuu,fill=uuuuuu,fill opacity=0.1] (0,0) -- (7,0) -- (7,7) -- (0,7) -- cycle;
\draw [line width=1.pt,color=black] (0,0)--(7,0)--(7,7)--(0,7)--cycle;
\draw [->,line width=1.pt] (0,0) -- (3,0);
\draw [->,line width=1.pt] (0,0) -- (0,3);
\draw [->,line width=1.pt] (7,7) -- (4,7);
\draw [->,line width=1.pt] (7,7) -- (7,4);
\draw [->,line width=1.pt](0,0) -- (1.5,1.5);
\draw [->,line width=1.pt](7,7) -- (5.5,5.5);
\draw[line width=2pt,color=black](0,0)--(7,7);
\draw[color=black] (2.25,1.25) node {$\scriptscriptstyle\w$};
\draw[color=black] (3,-1) node {$\scriptscriptstyle\aa_1$};
\draw[color=black] (-1,3) node {$\scriptscriptstyle\aa_1'$};
\draw[color=black] (4,8) node {$\scriptscriptstyle\aa_0$};
\draw[color=black] (8,4) node {$\scriptscriptstyle\aa_0'$};
\end{tikzpicture}
&
\begin{tikzpicture}[x=0.25cm,y=0.25cm,>=angle 60]
\fill[line width=2.pt,color=uuuuuu,fill=uuuuuu,fill opacity=0.1] (0,0) -- (7,0) -- (7,7) -- (0,7) -- cycle;
\draw [line width=1.pt,color=black] (0,0)--(7,0)--(7,7)--(0,7)--cycle;
\draw [->,line width=1.pt] (0,0) -- (3,0);
\draw [->,line width=1.pt] (0,0) -- (0,3);
\draw [->,line width=1.pt] (7,7) -- (4,7);
\draw [->,line width=1.pt] (7,7) -- (7,4);
\draw [->,line width=1.pt](0,0) -- (3,3);
\draw [->,line width=1.pt](7,7) -- (4,4);
\draw[line width=2pt,color=black](0,0)--(7,7);
\draw[color=black] (3.75,2.75) node {$\scriptscriptstyle\w$};
\draw[color=black] (3,-1) node {$\scriptscriptstyle\aa_1$};
\draw[color=black] (-1,3) node {$\scriptscriptstyle\aa_1'$};
\draw[color=black] (4,8) node {$\scriptscriptstyle\aa_0$};
\draw[color=black] (8,4) node {$\scriptscriptstyle\aa_0'$};
\end{tikzpicture}
&\begin{tikzpicture}[x=0.25cm,y=0.25cm,>=angle 60]
\fill[line width=2.pt,color=uuuuuu,fill=uuuuuu,fill opacity=0.1] (0,0) -- (7,0) -- (7,7) -- (0,7) -- cycle;
\draw [line width=1.pt,color=black] (0,0)--(7,0)--(7,7)--(0,7)--cycle;
\draw [->,line width=1.pt] (0,0) -- (3,0);
\draw [->,line width=1.pt] (0,0) -- (0,3);
\draw [->,line width=1.pt] (7,7) -- (5.5,7);
\draw [->,line width=1.pt] (7,7) -- (7,5.5);
\draw [->,line width=1.pt](0,0) -- (1.5,1.5);
\draw [->,line width=1.pt](7,7) -- (5.5,5.5);
\draw[line width=2pt,color=black](0,0)--(7,7);
\draw[color=black] (2.25,1.25) node {$\scriptscriptstyle\w$};
\draw[color=black] (3,-1) node {$\scriptscriptstyle\aa_1$};
\draw[color=black] (-1,3) node {$\scriptscriptstyle\aa_1'$};
\draw[color=black] (5.5,8) node {$\scriptscriptstyle\aa_0$};
\draw[color=black] (8,5.5) node {$\scriptscriptstyle\aa_0'$};
\end{tikzpicture}
\\
$\w=\frac12(\aa_1+\aa_1')$ & $\w={\aa_1+\aa_1'}$ & $\w=\frac12(\aa_1+\aa_1')$
        \end{tabular}
        \end{center}
\vspace{-9mm}\end{proof}
\begin{example}
  Consider the diagram $\sD_4^{(2)}(dd)$
  \[
    \hbox{\boxit{\begin{picture}(6000,2400)(-300,-1200)
            \put(0,0){\usebox{\vertex}}
            \put(0,0){\usebox{\leftbiedge}}
            \put(3600,0){\usebox{\rightbiedge}}
            \put(1800,0){\usebox{\dynkinatwo}}
            \multiput(0,0)(3600,0){2}{\usebox{\wcircle}}
            \multiput(1800,0)(3600,0){2}{\usebox{\wcircle}}
            \multiput(0,-250)(3600,0){2}{\line(0,-1){950}}
            \multiput(1800,250)(3600,0){2}{\line(0,1){950}}
            \put(0,-1200){\line(1,0){3600}}
            \put(1800,1200){\line(1,0){3600}}
          \end{picture}}}
  \]
  Then $k(\a_0^\vee)=k(\a_3^\vee)=1$ and
  $k(\a_1^\vee)=k(\a_2^\vee)=2$, and $\w_1=\aa_0+\aa_2$ and
  $\w_2=\aa_1+\aa_3$ and
  \[
    X_1=\frac23 P_{\a_1}+\frac13P_{\a_3},\
    X_2=\frac13 P_{\a_0}+\frac23P_{\a_2}.
  \]
\vspace{-4mm}
\newsavebox{\tetraeder}
\savebox{\tetraeder}{\includegraphics[width=30mm]{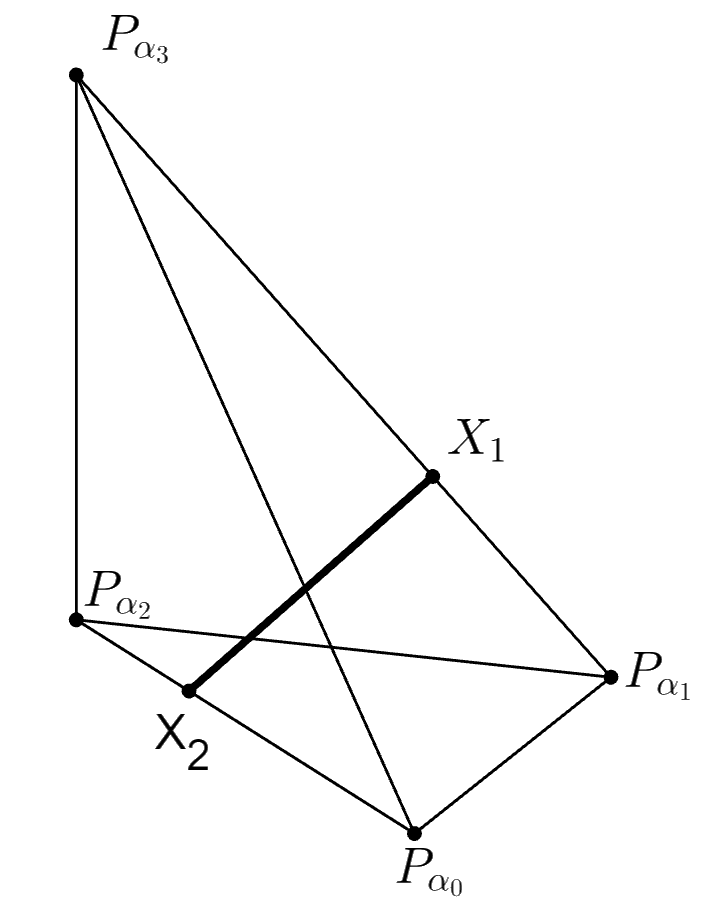}}
Here is a picture of $\P$ inside $\A$:\qquad\qquad\boxit{\begin{picture}(9000,12000)(4000,000)\put(0,0){\usebox\tetraeder}\end{picture}}
\end{example}
\begin{rems} \begin{enumerate}

  \item The three primitive diagrams for $\sA_1^{(1)}$ (see
    \cref{ex:a11a22}) correspond to the manifolds $S^4$ (the so-called
    \l/spinning 4-sphere\r/, \cite{HJ00,AMW02}), $S^2\times S^2$, and
    $\PP^2(\C)$, respectively (see \cite[\S2.7]{Kno14} for details).
    
 \item Generalizing the example above, the diagram
    $\sA_{n-1}^{(1)}(\Ca\Ca)$ with $n\ge2$ corresponds to the
    \l/spinning $2n$-sphere\r/ $S^{2n}$ discovered by
    Hurtubise-Jeffrey-Sjamaar \cite{HJS06}.
    
  \item The cases $\sC_{n\ge2}^{(1)}(\Cc\Cc)$ are realized by
    $\SP(2n)$ acting on the quaternionic Grassmannians
    $M=\operatorname{Gr}_d(\mathbb{H}^{n+1})$ (see \cite[Thm.\
    2.7.2]{Kno14}). This is a generalization of a result by Eshmatov
    \cite{Esh09} for $d=1$.
\end{enumerate}
\end{rems}
One can combine the classification of primitive diagrams with the
Reduction \cref{lemma:reduction}. For the formulation of the lemma, we define
\[
  k^\vee(S^c):=\gcd\{k(\a^\vee)\mid\a\in S^c\}.
\]
in case $\Phi$ is an irreducible affine root system.
\begin{defn}\label{def:primdia}
  Assume $\Phi$ is finite or irreducible. A \emph{spherical diagram}
  is a $6$-tuple $(S,S^c,S_1',c_1,S_2',c_2)$ with:
  \begin{enumerate}
    \item $S_1',S_2',S^c\subseteq S$ are
      pairwise disjoint
    \item $(S_{12},S_1',c_1,S_2',c_2)$ is a primitive diagram where
      $S_{12}$ is the connected closure of $S_1'\cup S_2'$ in
      $S\setminus S^c$. Set $\cD_i=(S_i,S_i',c_i)$ where $S_i$ is
      the connected closure of $S_i'$ in $S_{12}\setminus S_j'$.
    \item\label{it:primdia3} If $\cD_i$ is
      homogeneous then $\<w_{\cD_i},\aa^\vee\>\in\Z$ for all
      $\a\in S^c$.
    \item\label{it:primdia4} Assume $\cD_1$ and $\cD_2$ are both
      inhomogeneous with $\a_i:=\a_{\cD_i}$. Assume also that $\Phi$
      is affine and irreducible. Then $k^\vee(S^c)$ divides
      $k(\a_1^\vee)-k(\a_2^\vee)$.
    \end{enumerate}
\end{defn}
\begin{rem}
  The condition \ref{it:primdia3} is only relevant if $\cD_i$ is of type
  $\2d_{n\ge2}$ or $\2b_3'$.
\end{rem}
Again, the point of the definition is:
  
\begin{lem}\label{lem:SSScSc}
  Let $(X_1,X_2,\w)$ be a spherical triple. Put
  $S^c:=S\setminus(S(X_1)\cup S(X_2))$ and let $(S_i,S_i',c_i)$ be the
  local diagram at $X_i$. Then $(S,S^c,S_1',c_1,S_2',c_2)$ is a
  spherical diagram.
\end{lem}
\begin{proof}
  Only \ref{it:primdia4} needs an argument. The weight $\w$ satisfies
  $\<\w,\aa_{\cD_1}^\vee\>=1$, $\<\w,\aa_{\cD_2}^\vee\>=-1$,
  $\<\w,\aa^\vee\>\in\Z$ for $\a\in S^c$, and $\<\w,\aa^\vee\>=0$ for
  all other $\a\in S$. Hence \ref{it:primdia4} follows from the linear
  dependence relation \eqref{eq:kvee}.
\end{proof}
  \begin{thm}
  Let $(\Phi,\Xi)$ be simply connected and $\cD$ a non-empty spherical
  diagram on $\Phi$. Let $\Phi$ be finite or affine, irreducible. Then
  $\cD$ is realized by a spherical triple $(X_1,X_2,\w)$.
\end{thm}
\begin{proof}
  The two last conditions \ref{it:primdia3} and \ref{it:primdia4} of
  \cref{def:primdia} make sure that there exists $\w\in\Xi$ which
  gives rise to the appropriate local model at $X_i$ (for
  \ref{it:primdia4}, see the argument in the proof of
  \cref{lem:SSScSc}). We show the existence of a matching polytope
  first in the finite case. In this case, we may assume that all roots
  $\a=\aa$ are linear and $S$ is linearly independent. Additionally to
  the inequalities \eqref{eq:X1} for $\a\not\in S^c$ we get the
  inequalities $\a(X_1),\a(X_2)>0$ for $\a\in S^c$. Because of linear
  independence, the values $\a(X_1)$ with $\a\in S$ can be prescribed
  arbitrarily. By \cref{thm:real}, we can do that in such a way that
  all inequalities for $\a\not\in S^c$ are satisfied. Now we choose
  $\a(X_1)>\!\!>0$ for all $\a\in S^c$. Since $c$ in \eqref{eq:c} is
  not affected by this choice, this yields $\a(X_2)>0$ for
  all $\a\in S^c$, as well.
  
  Now let $\Phi$ be affine, irreducible. If $S^c=\leer$ then the
  spherical diagram would be in fact primitive. This case has been
  already dealt with. So let $S^c\ne\leer$ and fix $\a_0\in S^c$. Then
  $S^f:=S\setminus\{\a_0\}$ generates a finite root system. We may
  assume that all $\a\in S^f$ are linear. Then the existence of a
  solution $(X_1,X_2)$ satisfying all inequalities for all $\a\in S^f$
  has been shown above. Now observe that the set of these solutions
  form a cone. If we choose a solution sufficiently close to the
  origin we get $\a_0(X_1),\a_0(X_2)>0$.
\end{proof}
   
A spherical diagram is drawn like a primitive diagram where the
roots $\a\in S^c$ are indicated by circling them.
\begin{example}
  Consider the diagram on $\sD_7$:
  \[
    \hbox{\boxit{\begin{picture}(8700,3000)(0,-1500)
      \put(0,0){\usebox{\dynkinafive}}
      \put(7200,0){\usebox{\athreebifurc}}
      \put(5000,-300){$\bt$} \put(3600,0){\usebox{\wcircle}}
    \end{picture}}}
  \]
  Then $S^c=\{\a_3\}$, $S_1=\{\a_5,\a_6,\a_7\}$ and $S_2=\{\a_4,\a_5\}$.
\end{example}
\begin{example}
  In addition to the primitive diagrams of \cref{ex:a11a22}, the root
  system $\sA_1^{(1)}$ supports the following spherical diagrams:
  \[
    \begin{picture}(2400,1800)(-300,-900)
      \put(-400,-300){$\bt$} \put(0,0){\usebox\leftrightbiedge}
      \put(1800,0){\usebox{\wcircle}}
    \end{picture}\qquad
    \begin{picture}(2400,1800)(-300,-900) \put(0,0){\usebox{\aone}}
      \put(0,0){\usebox\leftrightbiedge}
      \put(1800,0){\usebox{\wcircle}}
    \end{picture}\qquad
    \begin{picture}(2400,1800)(-300,-900) \put(0,0){\usebox{\aprime}}
      \put(0,0){\usebox\leftrightbiedge}
      \put(1800,0){\usebox{\wcircle}}
    \end{picture}\qquad
    \begin{picture}(2400,1800)(-300,-900) \put(0,0){\usebox{\wcircle}}
      \put(0,0){\usebox\leftrightbiedge}
      \put(1500,-300){$\bt$}
    \end{picture}\qquad
    \begin{picture}(2400,1800)(-300,-900) \put(0,0){\usebox{\wcircle}}
      \put(0,0){\usebox\leftrightbiedge}
      \put(1800,0){\usebox{\aone}}
    \end{picture}\qquad
    \begin{picture}(2400,1800)(-300,-900) \put(0,0){\usebox{\wcircle}}
      \put(0,0){\usebox\leftrightbiedge}
      \put(1800,0){\usebox{\aprime}}
    \end{picture}\qquad
    \begin{picture}(2400,1800)(-300,-900) \put(0,0){\usebox{\wcircle}}
      \put(0,0){\usebox\leftrightbiedge}
      \put(1800,0){\usebox{\wcircle}}
    \end{picture}
  \]
  If one identifies the alcove $\A$ with the interval $[0,1]$ then
  $\P=[t,1]$ in the first three cases, $\P=[0,t]$ in cases 4 through 6
  and $\P=[t_1,t_2]$ in the last case where $0<t<1$ and $0<t_1<t_2<1$
  are arbitrary.
\end{example}
\begin{example}
  Up to automorphisms, all spherical diagrams supported on $\sA_2^{(1)}$ are
  listed in the top row of
  
  {\setlength{\unitlength}{0.006 pt}
\begin{longtable}{cccccccccccc}
\begin{picture}(4600,4050)(-300,-300)
\put(0,0) {\circle*{400}}
\put(4000,0) {\circle*{400}}
\put(2000,3464) {\circle*{400}}
\put(0,0){\line(1,0){4000}}
\put(0,0){\line(1,1.71) {2000}}
\put(4000,0){\line(-1,1.71) {2000}}
\put (0,0) {\circle{800}}
\put (4000,0) {\circle{800}}
\put (2000,3464) {\circle{800}}
\end{picture}
&
\begin{picture}(4600,4050)(-300,-300)
\put(0,0) {\circle*{400}}
\put(4000,0) {\circle*{400}}
\put(2000,3464) {\circle*{400}}
\put(0,0){\line(1,0){4000}}
\put(0,0){\line(1,1.71) {2000}}
\put(4000,0){\line(-1,1.71) {2000}}
\put (0,0) {\circle{800}}
\put (2000,3464) {\circle{800}}
\end{picture}
&
\begin{picture}(4600,4050)(-300,-300)
\put(0,0) {\circle*{400}}
\put(4000,0) {\circle*{400}}
\put(2000,3464) {\circle*{400}}
\put(0,0){\line(1,0){4000}}
\put(0,0){\line(1,1.71) {2000}}
\put(4000,0){\line(-1,1.71) {2000}}
\put (0,0) {\circle{800}}
\put (2000,3464) {\circle{800}}
\put (3600,-400) {\usebox\quadrat}
\end{picture}
&
\begin{picture}(4600,4050)(-300,-300)
\put(0,0) {\circle*{400}}
\put(4000,0) {\circle*{400}}
\put(2000,3464) {\circle*{400}}
\put(0,0){\line(1,0){4000}}
\put(0,0){\line(1,1.71) {2000}}
\put(4000,0){\line(-1,1.71) {2000}}
\put (0,0) {\circle{800}}
\put (2000,3464) {\circle{800}}
\put(4000,800) {\circle{800}}
\put(4000,-800) {\circle{800}}
\end{picture}\vspace{4pt}
&
\begin{picture}(4600,4050)(-300,-300)
\put(0,0) {\circle*{400}}
\put(4000,0) {\circle*{400}}
\put(2000,3464) {\circle*{400}}
\put(0,0){\line(1,0){4000}}
\put(0,0){\line(1,1.71) {2000}}
\put(4000,0){\line(-1,1.71) {2000}}
\put (0,0) {\circle{800}}
\put (2000,3464) {\circle{800}}
\put(4000,-800) {\circle{800}}
\end{picture}
&
\begin{picture}(4600,4050)(-300,-300)
\put(0,0) {\circle*{400}}
\put(4000,0) {\circle*{400}}
\put(2000,3464) {\circle*{400}}
\put(0,0){\line(1,0){4000}}
\put(0,0){\line(1,1.71) {2000}}
\put(4000,0){\line(-1,1.71) {2000}}
\put (2000,3464) {\circle{800}}
\linethickness{1pt}
\put(-400,-400){\usebox\quadrat}
\end{picture}
&
\begin{picture}(4600,4050)(-300,-300)
\put(0,0) {\circle*{400}}
\put(4000,0) {\circle*{400}}
\put(2000,3464) {\circle*{400}}
\put(0,0){\line(1,0){4000}}
\put(0,0){\line(1,1.71) {2000}}
\put(4000,0){\line(-1,1.71) {2000}}
\put (0,0) {\circle{800}}
\put (4000,0) {\circle{800}}
\put (2000,3464) {\circle{800}}
\put (300,200) {\line(1,1){500}}
\multiput (1100,400)(600,0){4} {\line(1,1){300}}
\multiput (800,700) (600,0){4} {\line(1,-1){300}}
\put(3200,700) {\line(1,-1){500}}
\end{picture}
&
\begin{picture}(4600,4050)(-300,-300)
\put(0,0) {\circle*{400}}
\put(4000,0) {\circle*{400}}
\put(2000,3464) {\circle*{400}}
\put(0,0){\line(1,0){4000}}
\put(0,0){\line(1,1.71) {2000}}
\put(4000,0){\line(-1,1.71) {2000}}
\put (2000,3464) {\circle{800}}
\linethickness{1pt}
\put(-400,-400){\usebox\quadrat}
\put(3600,-400){\usebox\quadrat}
\end{picture}
&
\begin{picture}(4600,4050)(-300,-300)
\put(0,0) {\circle*{400}}
\put(4000,0) {\circle*{400}}
\put(2000,3464) {\circle*{400}}
\put(0,0){\line(1,0){4000}}
\put(0,0){\line(1,1.71) {2000}}
\put(4000,0){\line(-1,1.71) {2000}}
\put (2000,3464) {\circle{800}}
\put(-400,-400){\usebox\quadrat}
\put(4000,800) {\circle{800}}
\put(4000,-800) {\circle{800}}
\end{picture}
&
\begin{picture}(4600,4050)(-300,-300)
\put(0,0) {\circle*{400}}
\put(4000,0) {\circle*{400}}
\put(2000,3464) {\circle*{400}}
\put(0,0){\line(1,0){4000}}
\put(0,0){\line(1,1.71) {2000}}
\put(4000,0){\line(-1,1.71) {2000}}
\put (0,0) {\circle{800}}
\put (4000,0) {\circle{800}}
\put(2000,2620) {\circle{800}}
\put(2000,4220) {\circle{800}}
\put (300,200) {\line(1,1){500}}
\multiput (1100,400)(600,0){4} {\line(1,1){300}}
\multiput (800,700) (600,0){4} {\line(1,-1){300}}
\put(3200,700) {\line(1,-1){500}}
\end{picture}
&
\begin{picture}(4600,4050)(-300,-300)
\put(0,0) {\circle*{400}}
\put(4000,0) {\circle*{400}}
\put(2000,3464) {\circle*{400}}
\put(0,0){\line(1,0){4000}}
\put(0,0){\line(1,1.71) {2000}}
\put(4000,0){\line(-1,1.71) {2000}}
\linethickness{1pt}
\put(-400,-400){\usebox\quadrat}
\put(3600,-400){\usebox\quadrat}
\end{picture}
\\
\hline
\begin{picture}(4000,4500)(0,0)
\put(0,0){\line(1,0){4000}}
\put(0,0){\line(1,1.71) {2000}}
\put(4000,0){\line(-1,1.71) {2000}}
\linethickness{1pt}
\put(1200,1000) {\line(4,1){1200}}
\end{picture}
&
\begin{picture}(4000,4500)(0,0)
\put(0,0){\line(1,0){4000}}
\put(0,0){\line(1,1.71) {2000}}
\put(4000,0){\line(-1,1.71) {2000}}
\linethickness{1pt}
\put(700,1212) {\line(1,1.7732){600}}
\end{picture}
&
\begin{picture}(4000,4500)(0,0)
\put(0,0){\line(1,0){4000}}
\put(0,0){\line(1,1.71) {2000}}
\put(4000,0){\line(-1,1.71) {2000}}
\thicklines
\linethickness{1pt}
\put(1300,2200){\line(0,-1){1400}}
\end{picture}
&
\begin{picture}(4000,4500)(0,0)
\put(0,0){\line(1,0){4000}}
\put(0,0){\line(1,1.71) {2000}}
\put(4000,0){\line(-1,1.71) {2000}}
\thicklines
\linethickness{1pt}
\put(1300,2252){\line(1,-0.5773){1200}}
\end{picture}
&
\begin{picture}(4000,4500)(0,0)
\put(0,0){\line(1,0){4000}}
\put(0,0){\line(1,1.71) {2000}}
\put(4000,0){\line(-1,1.71) {2000}}
\thicklines
\linethickness{1pt}
\put(1300,2252){\line(1,-0.5773){1200}}
\end{picture}
&
\begin{picture}(4000,4500)(0,0)
\put(0,0){\line(1,0){4000}}
\put(0,0){\line(1,1.71) {2000}}
\put(4000,0){\line(-1,1.71) {2000}}
\linethickness{1pt}
\put(2000,3464) {\line(-1,-1.732){1200}}
\end{picture}
&
\begin{picture}(4000,4500)(0,0)
\put(0,0){\line(1,0){4000}}
\put(0,0){\line(1,1.71) {2000}}
\put(4000,0){\line(-1,1.71) {2000}}
\linethickness{1pt}
\put(2000,3300) {\line(0,-1){2400}}
\end{picture}
&
\begin{picture}(4000,4500)(0,0)
\put(0,0){\line(1,0){4000}}
\put(0,0){\line(1,1.71) {2000}}
\put(4000,0){\line(-1,1.71) {2000}}
\linethickness{1pt}
\put(1000,1710) {\line(1,0){2000}}
\end{picture}
&
\begin{picture}(4000,4500)(0,0)
\put(0,0){\line(1,0){4000}}
\put(0,0){\line(1,1.71) {2000}}
\put(4000,0){\line(-1,1.71) {2000}}
\linethickness{1pt}
\put(1300,2252) {\line(1,-0.5773){2100}}
\end{picture}
&
\begin{picture}(4000,4500)(0,0)
\put(0,0){\line(1,0){4000}}
\put(0,0){\line(1,1.71) {2000}}
\put(4000,0){\line(-1,1.71) {2000}}
\linethickness{1pt}
\put(2000,0) {\line(0,1){3300}}
\end{picture}
&
\begin{picture}(4000,4500)(0,0)
\put(0,0){\line(1,0){4000}}
\put(0,0){\line(1,1.71) {2000}}
\put(4000,0){\line(-1,1.71) {2000}}
\linethickness{1pt}
\put(0,0) {\line(1,0){4000}}
\end{picture}
\\
\end{longtable}}
The bottom row lists the corresponding momentum polytopes.  Observe
that each simple root of a Dynkin diagram in the first row corresponds
to the opposite edge of the alcove below it.
\end{example} 
\begin{rem}
  If $\Phi$ is a product of more than one affine root system then
  there are problems with the metric of $\A$ as we have already seen
  for the primitive case where $\A$ must be a metric square. We have
  not explored this case in full generality.
\end{rem}
\section{Verification of \cref{thm:main}}
Recall $i\in\{1,2\}$ and $j:=3-i$. Let $(S,S_1',c_1,S_2',c_2)$ be a
primitive diagram. Recall that $S_i$ is the connected closure of
$S_i'$ in $S\setminus S_j'$. 

Put
\[
  S^p_0:=S_1\cap S_2.
\]
Observe that these are exactly the simple roots that are not decorated in the diagram.
Our strategy is to construct $S$ by gluing $S_1$ and $S_2$ along
$S_0^p$. For this we have to make sure that $S_i$ remains the
connected closure of $S_i'$.
\begin{lem}\label{lem:SSSS}
  Let $S$ be a graph with subsets $S_1',S_1,S_2',S_2$ such that
  $S_i'\subseteq S_i\subseteq S\setminus S_j'$. Assume
  that $S=S_1\cup S_2$ and $S_i=C(S_i',S_i)$. Then the following are
  equivalent:
  \begin{enumerate}
  \item\label{it:SSSS1} $S_i$ is the connected closure of $S_i'$ in
    $S\setminus S_j$.
  \item\label{it:SSSS2} The following two conditions hold:
    \begin{enumerate}
    \item\label{it:SSSS21} $S_1\cap S_2$ is the union of connected
      components of $S_i\setminus S_i'$.
    \item\label{it:SSSS22} If $\a_1\in S_1\setminus S_2$ is
      connected to $\a_2\in S_2\setminus S_1$ in $S$ then
      $\a_1\in S_1'$ and $\a_2\in S_2'$.
    \end{enumerate}
  \end{enumerate}
\end{lem}
\begin{proof}
  Because of $S_i=C(S_i',S_i)$, the assertion
  $S_i=C(S_i',S\setminus S_j')$ means that there are no edges between
  $S_i$ and
  \[
    S\setminus(S_i\cup S_j')=(S_j\setminus S_j')\setminus(S_i\cap
    S_j)=(S_j\setminus S_i)\setminus S_j'.
  \]
  This statement breaks up into two parts: There are no edges between
  $S_i\cap S_j$ and $(S_j\setminus S_j')\setminus(S_i\cap S_j)$ which
  is just condition \ref{it:SSSS2}\ref{it:SSSS21}. And there are no
  edges between $S_i\setminus S_j$ and
  $(S_j\setminus S_i)\setminus S_j'$ which is just condition
  \ref{it:SSSS2}\ref{it:SSSS22}.
\end{proof}
\subsection{The case $S_0^p\ne\leer$}\label{ss:nl}
We start our classification with:
\begin{lem}
  Let $\cD$ be a primitive spherical diagram with $S_0^p\ne\leer$ such that there
  is at least one edge between $S_1'$ and $S_2'$. Then
  $\cD\cong \sA_{n\ge2}^{(1)}(\Ca\Ca)$:
    \[
      \begin{picture}(9900,1200)(-450,-900)
        \thicklines
        \multiput(0,0)(5400,0){2}{\usebox{\dynkinathree}}
        \put(3600,0){\usebox{\shortsusp}}
        \multiput(-600,0)(9600,0){2}{\line(1,0){600}}
        \multiput(-600,0)(10200,0){2}{\line(0,-1){1200}} \put
        (-650,-1200){\line(1,0){10250}}
        \put(-350,-300){$\btr$} \put(8650, -300){$\btl$}
      \end{picture}
    \]
\end{lem}
\begin{proof}
  Let $\a\in S_0^p$ and let there be an edge between $\a_1\in S_1'$
  and $\a_2\in S_2'$. Inspection of \cref{table:local} implies that
  there are paths from $\a_1$ and $\a_2$ to $\a$
  respectively. Together with the edge they form a cycle in $S$ which
  implies that $S$ is of type $\sA_{n\ge2}^{(1)}$ with $\a_1,\a_2$ being
  adjacent. Therefore the local diagrams are either of type $a_m$ or
  $\Ca_m$. Since only diagrams of the latter type can be glued such
  that $S_0^p\ne\leer$ and $S$ is a cycle we get
  $\sA_{n\ge2}^{(1)}(\Ca\Ca)$ as the only possibility.
\end{proof}
Thus, we may assume from now on that $S$ is the union of $S_1$ and
$S_2$ minimally glued along $S_0^p$, i.e., with no further edges
added. To classify these diagrams we proceed by the type of
$S_0^p$. Helpful is the following table which lists for all isomorphy
types of $S_0^p$ the possible candidates for $S_1$ and $S_2$. The
factor $c$ is suppressed.
\begin{lis}\label{table:S0p}{Gluing data}
\begin{longtable}%
{|>{$}l<{$}|>{$}l<{$}>{$}l<{$}>{$}l<{$}>{$}l<{$}>{$}l<{$}>{$}l<{$}>{$}l<{$}>{$}l<{$}|}
    \hline
    S_0^p &\multicolumn{8}{l|}{Candidates for $S_1$ and $S_2$}\\
    \hline
    \sA_1&a_3&\Ca_2&b_2&c_{n\ge3}&c_3&d_3&g_2&\Cc_2\\
    \sA_2&a_4&\Ca_3&b'_3&&&&&\\
    \sA_3&a_5&\Ca_4&d_4&&&&&\\
    \sA_{n\ge4}&a_{n+2}&\Ca_{n+1}&&&&&&\\
    \hline
    \sB_2&b_3&c_4&\Cc_3&&&&&\\
    \sB_3&b_4&f_4&&&&&&\\
    \sB_{n\ge4}&b_{n+1}&&&&&&&\\
    \hline
    \sC_{n\ge3}&c_{n+2}&\Cc_{n+1}&&&&&&\\
    \hline
    \sD_4&d_5&&&&&&&\\
    \sD_{n\ge5}&d_{n+1}&&&&&&&\\
    \hline
    \sA_1\sA_1&c_3&d_3&&&&&&\\
    \sA_1\sC_{n\ge2}&c_{n+2}&&&&&&&\\
    \hline
  \end{longtable}
\end{lis}
\begin{rem} For $c_3$, the graph $S\setminus S'$ consists of two
  $\sA_1$-components. Therefore $c_3$ is listed twice in the
  $\sA_1$-row. Also the case $S_0^p=\sD_4$ is listed separately since
  $\sD_4$ has automorphisms which don't extend to $\sD_5$.
\end{rem}
Using the table, it easy to find all primitive triples with
$S_0^p\ne\leer$. Since the case-by-case considerations are lengthy we
just give an instructive example namely where $S_0^p=\sA_2$. Here the
following nine combinations have to be considered:
\[
  a_4\cup a_4,\ a_4\cup\Ca_3,\ \Ca_3\cup\Ca_3 (2\times),\ a_4\cup b_3',\
  \Ca_3\cup b_3' (2\times),\ b_3'\cup b_3' (2\times).
\]
We omitted the possibility of a factor of $\3$ for
$b_3'$. Observe that in three cases two different gluings are
possible. It turns out that all cases lead to a valid spherical
diagram (namely $\sD_5^{(1)}(aa)$, $\sD_5(a\Ca)$, $\sD_4(\Ca\Ca)$, $\sA_4(\Ca\Ca')$, $\sB_4(b'\Ca)$ and $\sD_3^{(2)}(b'b')$ in the notation of \cref{ListHom}) except for $a_4\cup b_3'$ and one of the gluings of
$\Ca_3\cup b_3'$ which do not lead to affine Dynkin diagrams.

\subsection{The case $S_0^p=\leer$}
Here, according to \cref{lem:SSSS}, $S$ is the disjoint union of $S_1$
and $S_2$ stitched together with edges between $S_1'$ and $S_2'$.

A rather trivial subcase is when $S_2=S_2'=\leer$. Then $\cD$ is just a local diagram all of which appear in \cref{ListHom}.

Therefore, assume from now on that $S_1,S_2\ne\leer$. Since then
$1\le|S_i'|\le 2$, the number $N$ of edges between $S_1'$ and $S_2'$
is at most $4$. This yields $5$ subcases.

\subsubsection{$N=0$}\label{sss:0}
In this case, $S$ is the disconnected union of $S_1$ and $S_2$. If the
triple were bihomogeneous we cannot have $\w_{\cD_1}+\w_{\cD_2}=0$. If
$\cD_1$ were homogeneous and $\cD_2$ inhomogeneous then
$\<\w_{\cD_1},\a_{\cD_2}^\vee\>=0\ne-1$. Therefore, the triple is
bi-inhomogeneous and we get the three items
$\sA_{m\ge1}\times\sA_{n\ge1}$, $\sA_{m\ge1}\times\sC_{n\ge2}$, and
$\sC_{m\ge2}\times\sC_{n\ge2}$ near the end of the table.
\subsubsection{$N=1$}\label{sss:1}
In this case the diagram $\cD$ is the disjoint union of two local
diagrams connected by one edge between some $\a_1\in S_1'$ and
$\a_2\in S_2'$. One can now go through all pairs of local diagrams and
all possibilities for the connecting edge. This works well if one or
both local diagrams are of type $(d_2)$. Otherwise, it is easier to go
through all possible connected Dynkin diagrams for $S$ and omit one of
its edges. The remaining diagram admits very few possibilities for
$\cD_1$ and $\cD_2$. This way one can check easily that the table is
complete with respect to this subcase.  Let's look, e.g., at
$S=\sF_4^{(1)}$
  \[
 \begin{picture}(7200,1100)(0,-300)
 \put(1800,0){\usebox\dynkinffour}
 \put(0,0){\usebox\dynkinatwo}
 \put(-200,300){\tiny1}
 \put(1600,300){\tiny2}
 \put(3400,300){\tiny3}
 \put(5200,300){\tiny4}
 \put(6900,300){\tiny2}
\end{picture}
\]
Omitting one edge yields
\[
    \begin{picture}(7200,1100)(0,-300)
 \put(1800,0){\usebox\dynkinffour}
 \put(0,0){\usebox\dynkinaone}
 \put(0,0){\circle{600}}
 \put(1800,0){\circle{600}}
\end{picture}\qquad
    \begin{picture}(7200,1100)(0,-300)
 \put(5400,0){\usebox\dynkinatwo}
 \put(3600,0){\usebox\dynkinbtwo}
 \put(0,0){\usebox\dynkinatwo}
 \put(1800,0){\circle{600}}
 \put(3600,0){\circle{600}}
\end{picture}\qquad
    \begin{picture}(7200,1100)(0,-300)
 \put(5400,0){\usebox\dynkinatwo}
 \put(0,0){\usebox\dynkinathree}
 \put(3600,0){\circle{600}}
 \put(5400,0){\circle{600}}
\end{picture}\qquad
    \begin{picture}(7200,1100)(0,-300)
 \put(7200,0){\usebox\dynkinaone}
 \put(0,0){\usebox\dynkinbfour}
 \put(5400,0){\circle{600}}
 \put(7200,0){\circle{600}}
\end{picture}
\]  
Each component has to support a local diagram such that the circled
vertex is in $S'$. This rules out all cases except the third one
where the local diagrams could be of type $a_n$ or $\Ca_n$. This
yields
\[
 \begin{picture}(7200,1100)(0,-300)
 \put(1800,0){\usebox\dynkinffour}
 \put(0,0){\usebox\dynkinatwo}
 \put(3250,-300){\symb}
 \put(5050,-300){\symb}
\end{picture}\qquad
 \begin{picture}(7200,1100)(0,-300)
 \put(1800,0){\usebox\dynkinffour}
 \put(0,0){\usebox\dynkinatwo}
\put(3250,-300){\symb}
\put(5400,0){\usebox\atwo}
\end{picture}\qquad
 \begin{picture}(7200,1100)(0,-300)
 \put(1800,0){\usebox\dynkinffour}
 \put(0,0){\usebox\dynkinatwo}
\put(0,0){\usebox\athree}
 \put(5050,-300){\symb}
\end{picture}\qquad
 \begin{picture}(7200,1100)(0,-300)
 \put(1800,0){\usebox\dynkinffour}
 \put(0,0){\usebox\dynkinatwo}
\put(0,0){\usebox\athree}
\put(5400,0){\usebox\atwo}
\end{picture}
\]
The first diagram violates \cref{def:prim5}\ref{it:prim5-3}\ref{it:prim5-33}, the second is primitive and is contained in \cref{ListHom}, the third violates \ref{it:prim5-3}\ref{it:prim5-32}, and the fourth violates \ref{it:prim5-3}\ref{it:prim5-31}.
\subsubsection{$N=2$}\label{sss:2}
In this case, at least one of the $S_i'$, say $S_1'$, has two different
elements $\a_1,\a_1'$ which are connected to elements
$\a_2,\a_2'\in S_2'$, respectively.

Assume first that $\a_2\ne\a_2'$. Then both local models are either
of type $a_{n\ge2}$ or $d_2$. One checks easily that this yields
the cases
\[
\sA_{n\ge3}^{(1)}(aa),\ 
\sC_{n\ge3}^{(1)}(ad),\ 
\sD_{n+1}^{(2)}(ad),\ 
\sA_1^{(1)}\times \sA_1^{(1)}(dd),\text{ or }
\sA_2^{(2)}\times \sA_2^{(2)}(dd).
\]
The second subcase is $\a_2=\a_2'$. If $S_1$ is of type $a_{n\ge2}$
one ends up with $\sA_{n\ge2}^{(1)}(aa)$, $d=1$. Otherwise $S_1$ is of
type $d_2$. Now one can go through all local diagrams for $S_2$ and
all possible edges between $\a_1,\a_1'$ and $\a_2\in S_2'$.
  
\subsubsection{$N=3$}\label{sss:3}
In this case, there are distinct elements $\a_1,\a_1'\in S_1'$ and
$\a_2,\a_2'\in S_2'$ which form a string
$\a_1,\a_2,\a_1',\a_2'$. Both local diagrams are of type $d_2$ which
leaves only the case $\sD_4^{(2)}(dd)$.
\subsubsection{$N=4$}\label{sss:4}
Here $\a_1,\a_2,\a_1',\a_2'$ form a cycle, so $S$ is of type
$\sA_3^{(1)}$. The local diagrams are both $d_2$. This yields
$\sA_3^{(1)}(dd)$.

This concludes the verification of the table.
\newpage

\section{Tables}
\begin{lis}\label{table:Dynkin}{Affine Dynkin diagrams and their labels}
\begin{longtable}{ll}
  $\sA_1^{(1)}$ &
  \boxit{\begin{picture}(1800,900)
 \put(0,0){\usebox\leftrightbiedge}
 \multiput(-300,300)(1800,0){2}{\tiny1}
\end{picture}} \\
$\sA_n^{(1)}, n\ge 2$ &
\boxit{\begin{picture}(9000,1800)
    \thicklines
    \put(0,0){\usebox{\dynkinathree}}
  \put(3600,0){\usebox\shortsusp}
  \put(5400,0){\usebox\dynkinathree}
  \put(4500,1800){\usebox\vertex}
  \line(4500,1800){4500}
  \put(4500,0){\line(-4500,1800){4500}}
  \multiput(-4750,300)(1800,0){6}{\tiny1}
  \put(-250,1000){\tiny1}
\end{picture}} \\
$\sB_n^{(1)}, n\ge 3$&
\boxit{\begin{picture}(9000,2700)(0,-1800)
\put(0,0){\usebox\dynkinathree}
\put(3600,0){\usebox\shortsusp}
\put(5400,0){\usebox\dynkinbthree}
\put(1800,0){\usebox\vedge}
\put(-150,300){\tiny1}
\put(1100,-1650){\tiny1}
\multiput(1600,300)(1800,0){5}{\tiny2}
\end{picture}}\\
$\sC_n^{(1)}, n \ge 2$ &
\boxit{\begin{picture}(7200,1150)(0,-300)
  \put(0,0){\usebox\rightbiedge}
  \put(1800,0){\usebox\susp}
  \put(5400,0){\usebox\leftbiedge}
  \put(-200,300){\tiny 1}
  \put(7000,300){\tiny 1}
  \multiput(1600,300)(3600,0){2}{\tiny 2}
\end{picture}} \\
$\sD_n^{(1)}, n\ge 4$&
\boxit{\begin{picture}(10800,3800)(0,-1900)
 \put(0,0){\usebox\dynkinathree}
 \put(3600,0){\usebox\susp}
 \put(7200,0){\usebox\dynkinathree}
 \put(1800,0){\usebox\vedge}
 \put(9000,1800){\usebox\vedge}
 \put(9000,1800){\usebox\vertex}
 \put(1950,-1650){\tiny1}
 \put(9200,1300){\tiny1}
 \put(-200,300){\tiny1}
 \put(10600,300){\tiny1}
 \multiput(1600,300)(1800,0){2}{\tiny2}
 \put(7000,300){\tiny2}
 \put(8500,300){\tiny2}
\end{picture}} \\
$\sE_6^{(1)}$ &
\boxit{\begin{picture}(7200,4500)(0,-3600)
 \put(0,0){\usebox\dynkinesix}
 \thicklines
 \put(3600,-1800){\line(0,-1){1800}}
 \put(3600,-3600){\usebox{\vertex}}
 \put(-300,300){\tiny1}
\put(1500,300){\tiny2}
\put(3300,300){\tiny3}
\put(5100,300){\tiny2}
\put(6900,300){\tiny1}
\put(3000,-1500){\tiny2}
\put(3000,-3300){\tiny1}
\end{picture}}\\
$\sE_7^{(1)}$ &
\boxit{\begin{picture}(10800,2700)(0,-1800)
    \put(0,0){\usebox\dynkinatwo}
    \put(1800,0){\usebox\dynkineseven}
 \thicklines
 \put(-300,300){\tiny1}
 \put(1500,300){\tiny2}
 \put(3300,300){\tiny3}
 \put(5100,300){\tiny4}
 \put(4800,-1500){\tiny2}
 \put(6900,300){\tiny3}
 \put(8700,300){\tiny2}
 \put(10500,300){\tiny1}
\end{picture}} \\
$\sE_8^{(1)}$ &
\boxit{\begin{picture}(12600,2700)(0,-1800)
 \put(0,0){\usebox\dynkineeight}
\thicklines
\put(10800,0){\line(1,0){1800}}
\put(12600,0){\usebox\vertex}
\put(-300,300){\tiny 2}
\put(1500,300){\tiny4}
\put(3300,300){\tiny6}
\put(5100,300){\tiny5}
\put(6900,300){\tiny4}
\put(8700,300){\tiny3}
\put(10600,300){\tiny2}
\put(12300,300){\tiny1}
\put(3000,-1500){\tiny3}
\end{picture}}\\
$\sF_4^{(1)}$&
\boxit{\begin{picture}(7200,1100)(0,-300)
 \put(1800,0){\usebox\dynkinffour}
 \put(0,0){\usebox\dynkinatwo}
 \put(-200,300){\tiny1}
 \put(1600,300){\tiny2}
 \put(3400,300){\tiny3}
 \put(5200,300){\tiny4}
 \put(6900,300){\tiny2}
\end{picture}}\\
$\sG_2^{(1)}$ &
\boxit{\begin{picture}(3600,1400)(0,-500)
 \put(0,0){\usebox\dynkinathree}
 \multiput(1800,150)(0,-300){2}{\line(1,0){1800}}
 \put(3300,0){\line(-1,1){500}}
  \put(3300,0){\line(-1,-1){500}}
  \put(-200,300){\tiny1}
  \put(1600,300){\tiny2}
  \put(3400,300){\tiny3}
\end{picture}} \\
$\sA_2^{(2)}$ &
\boxit{\begin{picture}(1800,1350)(0,-500)
\multiput(0,0)(1800,0){2}{\circle*{300}}
\multiput(0,-67)(0,133){2}{\line(1,0){1800}}
\multiput(0,-200)(0,400){2}{\line(1,0){1800}}
\multiput(150,0)(25,25){20}{\circle*{50}}
\multiput(150,0)(25,-25){20}{\circle*{50}}
\put(-200,300){\tiny 2}
\put(1600,300){\tiny 1}
\end{picture}} \\
$\sA_{2n}^{(2)}, n\ge 2$&
\boxit{\begin{picture}(7200,1200)(0,-300)
 \put(0,0){\usebox\vertex}
  \put(0,0){\usebox\leftbiedge}
  \put(1800,0){\usebox\susp}
  \put(5400,0){\usebox\leftbiedge}
  \multiput(-200,300)(1800,0){2}{\tiny 2}
  \put(5200,300){\tiny 2}
  \put(7000,300){\tiny 1}
 \end{picture}}\\
 $\sA_{2n-1}^{(2)}, n\ge 3$ &
 \boxit{\begin{picture}(10800,2700)(0,-1800)
 \put(0,0){\usebox\dynkinathree}
 \put(3600,0){\usebox\susp}
 \put(7200,0){\usebox\dynkinatwo}
 \put(9000,0){\usebox\leftbiedge}
 \put(1800,0){\usebox\vedge}
 \put(1950,-1650){\tiny1}
 \put(-200,300){\tiny1}
 \put(10600,300){\tiny1}
 \multiput(1600,300)(1800,0){2}{\tiny2}
 \put(7000,300){\tiny2}
 \put(8800,300){\tiny2}
\end{picture}} \\
$\sD_{n+1}^{(2)}, n\ge 2$ &
\boxit{\begin{picture}(7200,1200)(0,-350)
  \put(0,0){\usebox\leftbiedge}
  \put(0,0){\usebox\vertex}
  \put(1800,0){\usebox\susp}
  \put(5400,0){\usebox\rightbiedge}
  \put(-200,300){\tiny 1}
  \put(7000,300){\tiny 1}
  \multiput(1600,300)(3600,0){2}{\tiny 1}
\end{picture}}\\
$\sE_6^{(2)}$&
\boxit{\begin{picture}(7200,1200)(0,-350)
  \put(0,0){\usebox\dynkinathree}
  \put(3600,0){\usebox\leftbiedge}
  \put(5400,0){\usebox\dynkinatwo}
  \put(0,200){\tiny 1}
  \put(1600,300){\tiny 2}
  \put(3400,300){\tiny 3}
  \put(5200,300){\tiny 2}
  \put(7000,300){\tiny 1}
\end{picture}}\\
$\sD_4^{(3)}$ &
\boxit{\begin{picture}(3600,1350)(0,-500)
  \put(0,0){\usebox\dynkinathree}
   \multiput(1800,150)(0,-300){2}{\line(1,0){1800}}
 \put(2100,0){\line(1,1){500}}
  \put(2100,0){\line(1,-1){500}}
  \multiput(-200,300)(3600,0){2}{\tiny 1}
  \put(1600,300){\tiny 2}
 \end{picture}}
\end{longtable}
\end{lis}
\newpage
\begin{lis}\label{table:local}{Primitive local models}
  \begin{longtable}{ll<{\hspace{-1.5mm}}l<{\hspace{-1.5mm}}ll}
    \multicolumn{5}{l}{\bf Homogeneous models}\\
    &$L_0$&$H_0$&$\w$&Diagram\vrule depth 0pt height0pt width25pt\\
    \hline $a_1$&$\PGL(2)$&$\GL(1)$&$\a_1$&
    \boxit{\begin{picture}(600,1800)(-300,-900)
        \put(0,0){\usebox\aone}\end{picture}}\\
    $2a_1$&$\PGL(2)$&$N(\C^*)$&$2\a_1$&
    \boxit{\begin{picture}(600,900)(-300,-900)
        \put(0,0){\usebox\aprime}\end{picture}}\\
    $a_{n\ge2}$&$\PGL(n+1)$&$\GL(n)/\mu_{n+1}$&$\a_1+\ldots+\a_n$&
    \boxit{\begin{picture}(6000,600)(-300,-300)
        \put(0,0){\usebox\mediumam}\end{picture}}\\
    $b_{n\ge2}$&$\SO(2n+1)$&$\SO(2n)$&$\a_1+\ldots+\a_n$&
    \boxit{\begin{picture}(7500,600)(-300,-300)
        \put(0,0){\usebox\shortbm}\end{picture}}\\
    $2b_{n\ge2}$&$\SO(2n+1)$&$\O(2n)$&$2\a_1+\ldots+2\a_n$&
    \boxit{\begin{picture}(7500,1300)(-300,-300)
        \put(0,0){\usebox\shortbprimem}\end{picture}}\\
    $c_{n\ge3}$&$\PSP(2n)$&$\SP(2){\times^{\mu_2}}\SP(2n{-}2)$&$\a_1{+}2\a_2{+}\ldots{+}2\a_{n-1}{+}\a_n$&
    \boxit{\begin{picture}(9000,600)(0,-300)
        \put(0,0){\usebox\shortcm}\end{picture}}\\
    $\2d_{n\ge4}$&$\SO(2n)$&$\SO(2n-1)$&$\a_1{+}\ldots{+}\a_{n{-}2}{+}\2\a_{n{-}1}{+}\2\a_n$&
    \boxit{\begin{picture}(7000,2400)(-300,-1200)
        \put(0,0){\usebox\shortdm} \put (-600,500){\tiny$\3$}
      \end{picture}}\\
    $d_{n\ge4}$&$\PSO(2n)$&$\SO(2n-1)$&$2\a_1{+}\ldots{+}2\a_{n{-}2}{+}\a_{n{-}1}{+}\a_n$&
    \boxit{\begin{picture}(7000,2400)(-300,-1200)\put(0,0){\usebox\shortdm}\end{picture}}\\
    $\2d_2$&$\SO(4)$&$\SO(3)$&$\a+\a'$&
    \boxit{\begin{picture}(2400,1800)(-300,-300)
        \put(0,0){\circle{600}}
        \put(0,0){\usebox{\vertex}}
        \put(1800,0){\circle{600}}
        \put(1800,0){\usebox{\vertex}}
        \multiput(0,300)(1800,0){2}{\line(0,1){300}}
        \put(0,600){\line(1,0){1800}}
        \put(200,900){\tiny $\3$}
      \end{picture}}
    \\
    $d_2$&$\SO(3)\times\SO(3)$&$\SO(3)$&$\2\a+\2\a'$&
    \boxit{\begin{picture}(2400,900)(-300,-300)
        \put(0,0){\circle{600}}
        \put(0,0){\usebox{\vertex}}
        \put(1800,0){\circle{600}}
        \put(1800,0){\usebox{\vertex}}
        \multiput(0,300)(1800,0){2}{\line(0,1){300}}
        \put(0,600){\line(1,0){1800}}
      \end{picture}}
    \\
    $\2d_3$&$\SO(6)$&$\SO(5)$&$\2\a_1+\a_2+\2\a_3$&
    \boxit{\begin{picture}(3600,1500)(0,-300)
        \put(0,0){\usebox{\dynkinathree}}
        \put(1800,0){\circle*{600}} \put(1200,600){\tiny $\3$}
      \end{picture}}
    \\
    $d_3$&$\PSO(6)$&$\SO(5)$&$\a_1+2\a_2+\a_3$&
    \boxit{\begin{picture}(3600,600)(0,-300) \put(0,0){\usebox{\dynkinathree}}
        \put(1800,0){\circle*{600}}
      \end{picture}}
    \\
    $f_4$&$\sF_4$&$\SPIN(9)$&$\a_1+2\a_2+3\a_3+2\a_4$&
    \boxit{\begin{picture}(5700,600)(0,-300)\put(0,0){\usebox\ffour}\end{picture}}\\
    $g_2$&$\sG_2$&$SL(3)$&$2\a_1+\a_2$&
    \boxit{\begin{picture}(2100,600)(-300,-300)\put(0,0){\usebox\gtwo}\end{picture}}\\
    $2g_2$&$\sG_2$&$N(SL(3))$&$4\a_1+2\a_2$&
    \boxit{\begin{picture}(2100,1300)(-300,-300)\put(0,0){\usebox\gprimetwo}\end{picture}}\\
    $\2b_3'$&$\SPIN(7)$&$\sG_2$&$\2\a_1+\a_2+\frac32\a_3$&
    \boxit{\begin{picture}(3900,1400)(0,-300)
        \put(0,0){\usebox{\bthirdthree}} \put (2900,500){\tiny$\3$}
      \end{picture}}
    \\
    $b_3'$&$\SO(7)$&$\sG_2$&$\a_1+2\a_2+3\a_3$&
    \boxit{\begin{picture}(3900,600)(0,-300)\put(0,0){\usebox\bthirdthree}\end{picture}}\\
    \noalign{\vspace{10pt}}
    \multicolumn{5}{l}{\bf Inhomogeneous models}\\
    &$L_0$&$V_0$&$\w$&Diagram\\
    \hline
    $\Ca_0$&$\GL(1)$&$\C$&$\sim0$&\boxit{$\leer$}\\
    $\Ca_{n\ge1}$&$\GL(n+1)$&$\C^{n+1}$&$\sim\pi_1$&
    \boxit{\begin{picture}(9300,600)(-300,-300)
        \put(0,0){\usebox\dynkinatwo} \put(1800,0){\usebox\susp}
        \put(5400,0){\usebox\dynkinathree} \put(-350,-300){$\btr$}
      \end{picture}}
    \\
    $\Cc_{n\ge2}$&$\GSP(2n)$&$\C^{2n}$&$\sim\pi_1$&
    \boxit{\begin{picture}(9300,600)(-300,-300)
        \put(0,0){\usebox\dynkinatwo}
        \put(1800,0){\usebox\susp}
        \put(5400,0){\usebox\dynkincthree}
        \put(-350,-300){$\btr$}
      \end{picture}}
    \\
  \end{longtable}
\end{lis}
\newpage
\begin{lis}\label{ListHom}{Primitive diagrams}
      \tiny{
        \begin{longtable}{|c|>{\hspace{-1mm}}l<{\hspace{-2mm}}|p{55mm}Z|>{\hspace{-4pt}}c<{\hspace{-5pt}\vrule width 0pt}|l|p{15mm}D|}
          \hline
          $\Phi$&case&$\w=\overline\w_1=-\overline\w_2$ &$t$&factor& diagram&scope&\\\hline\endhead
          \noalign{\begin{center}\bf The empty case\end{center}}
\hline
          $\leer$&$(\leer)$&$\w\ne0$&&&$\leer$&&\ref{table:local}\\
\hline
\noalign{\begin{center}\bf The affine simple cases\end{center}}
\hline
          $\sA_{n\ge1}^{(1)}$&$(aa)$&\W\a_0+\ldots+\a_{d-1}|\a_d+\ldots+\a_n|&1&
          $[2]_{n=1}$ &
          \boxit{\begin{picture}(9900,1800)(-450,-900)
            \thicklines
            \multiput(0,0)(5400,0){2}{\usebox{\shortam}}
            \put(3600,0){\usebox{\edge}}
            \multiput(-450,0)(9250,0){2}{\line(1,0){600}}
            \multiput(-400,0)(9800,0){2}{\line(0,-1){900}}
            \put (-450,-900){\line(1,0){9850}}
            \put(-600,600){\tiny$\a_0$}
            \put(5000,600){\tiny$\a_d$}
          \end{picture}} &$1\le d\le n$ &\ref{sss:2}/\ref{sss:1}\\
&$(dd_2)$&\W\w_1=\a_0+\a_2|\w_2=\a_1+\a_3|&1&
          $\2$&\boxit{\begin{picture}(2400,2400)(-300,-300)
            \multiput(0,0)(0,1800){2}{\usebox{\edge}}
            \multiput(0,1800)(1800,0){2}{\usebox{\vedge}}
            \put(0,0){\usebox{\wcircle}}
            \put(1800,0){\usebox{\wcircle}}
            \put(0,1800){\usebox{\wcircle}}
            \put(1800,1800){\usebox{\wcircle}}
            \put(200,200){\line(1,1){1400}}
            \put(200,1600){\line(1,-1){1400}}
          \end{picture}}& &\ref{sss:4}\\
          &$(dd_3)$&\W\a_0+2\a_1+\a_2|\a_2+2\a_3+\a_0|&2&
          $\2$&\boxit{\begin{picture}(2400,2400)(-300,-300)
            \multiput(0,0)(0,1800){2}{\usebox{\edge}}
            \multiput(0,1800)(1800,0){2}{\usebox{\vedge}}
            \put(0,0){\circle*{600}} \put(1800,1800){\circle*{600}}
          \end{picture}}&&\ref{ss:nl}\\
          &$(\Ca\Ca)$&\WWW\pi_0-\pi_n|-\pi_0+\pi_n|&&&
          \boxit{\begin{picture}(9900,1200)(-450,-900)
            \thicklines
            \multiput(0,0)(5400,0){2}{\usebox{\dynkinathree}}
            \put(3600,0){\usebox{\shortsusp}}
            \multiput(-450,0)(9250,0){2}{\line(1,0){600}}
            \multiput(-400,0)(9800,0){2}{\line(0,-1){900}}
            \put (-450,-900){\line(1,0){9850}}
            \put(-350,-300){$\btr$} \put(8650, -300){$\btl$}
          \end{picture}}
          &&\ref{ss:nl}/\ref{sss:1}\\
          \hline
          $\sB_{n\ge3}^{(1)}$
          &$(bd)$&\W\a_0+\a_1+2\a_2+\ldots+2\a_{d-1}|2\a_d+\ldots+2\a_n|
          &1&$\2$&
          \boxit{\begin{picture}(10800,3000)(0,-1800)
            \thicklines
            \multiput(0,0)(1800,0){2}{\usebox{\vertex}}
            \multiput(1800,-1800)(1800,1800){2}{\usebox{\vertex}}
            \multiput(0,0)(1800,0){2}{\line(1,0){1800}}
            \put(1800,0){\line(0,-1){1800}}
            \put(3600,0){\usebox{\shortsusp}}
            \put(5400,0){\usebox{\edge}}
            \put(7200,0){\usebox{\shortsusp}}
            \put(9000,0){\usebox{\rightbiedge}}
            \put(7000,-1000){\tiny$\a_d$} \put(5400,0){\circle*{600}}
            \put(7200,0){\circle*{600}} \put(7000,600){\tiny 2}
          \end{picture}}&$2\le d\le n$ &\ref{sss:1}/\ref{sss:2}\\
          &$(bb)$&\W\a_1+\a_2+\ldots+\a_n|\a_0+\a_2+\ldots+\a_n|&1&$2$
          &\boxit{\begin{picture}(7600,2400)(-300,-2100)
            \thicklines
            \multiput(0,0)(1800,0){2}{\usebox{\vertex}}
            \put(1800,-1800){\usebox{\vertex}}
            \put(0,0){\line(1,0){1800}}
            \put(1800,0){\line(0,-1){1800}}
            \put(1800,0){\usebox{\susp}}
            \put(5400,0){\usebox{\rightbiedge}}
            \put(0,0){\circle*{600}} \put(1800,-1800){\circle*{600}}
          \end{picture}}&&\ref{ss:nl}\\
          &$(b'\Ca)$&\WW\2(\a_1+2\a_2+3\a_3)|\pi_0-\pi_3|&&&
          \boxit{\begin{picture}(4400,3200)(-100,-2000)
            \thicklines \put(0,0){\usebox\dynkinbthree}
            \put(1800,0){\usebox\vedge} \put(3600,0){\circle*{600}}
            \put(1500,-2000){$\bt$} \put(3000,600){\tiny $\3$}
          \end{picture}}&&\ref{ss:nl}\\
          &$(d\Ca)$&\WW\2(\a_1+\a_3)|-\pi_1+\pi_2-\pi_3|&&&
          \boxit{\begin{picture}(4300,3500)(-300,-1900)
               \put(0,0){\circle{600}}
            \put(3600,0){\circle{600}}
            \multiput(0,300)(3600,0){2}{\line(0,1){300}}
            \put(0,600){\line(1,0){3600}}
            \thicklines
            \put(0,0){\usebox\dynkinbthree}
            \put(1800,0){\usebox\vedge}
            \put(1400,800){\tiny$\3$}
            \put(1450,-300){$\btd$}
          \end{picture}} & &\ref{sss:2}\\
          \hline
          $\sC_{n\ge2}^{(1)}$
          &$(cc)$&\W\a_0+2\a_1+\ldots+2\a_{d-1}+\a_d|
          \a_d+2\a_{d+1}+\ldots+2\a_{n-1}+\a_n|&1&$[2]_{n=2}$
          &\boxit{\begin{picture}(11000,1200)(-100,-300)
            \put(5000,600){\tiny$\a_d$}
            \put(0,0){\usebox{\rightbiedge}}
            \put(9000,0){\usebox{\leftbiedge}}
            \put(1800,0){\usebox{\shortsusp}}
            \put(3600,0){\usebox{\dynkinathree}}
            \put(7200,0){\usebox{\shortsusp}}
            \multiput(3600,0)(3600,0){2}{\circle*{600}}
          \end{picture}}
          &$1\le d<n$& \\
          &$(ad)$&\W\2(\a_0+\a_n)|\a_1+\ldots+\a_{n-1}|
          &$\2$&$[2]_{n=2}$&
          \boxit{\begin{picture}(7800,1700)(-300,-1300) \put(3000,
            -1000){\tiny \3} \put(0,0){\usebox{\rightbiedge}}
            \put(5400,0){\usebox{\leftbiedge}}
            \put(1800,0){\usebox{\shortam}}
            \multiput(0,0)(7200,0){2}{\usebox{\wcircle}}
            \multiput(0,-250)(7200,0){2}{\line(0,-1){950}}
            \put(0,-1200){\line(1,0){7200}}
          \end{picture}}
          & &\ref{sss:2}\\
          &$(\Cc\Cc)$&\WWW\pi_{d-1}-\pi_d|-\pi_{d-1}+\pi_d|&&&
          \boxit{\begin{picture}(9200,1300)(-100,-300)
            \put(0,0){\usebox{\rightbiedge}}
            \put(7200,0){\usebox{\leftbiedge}}
            \put(1800,0){\usebox{\shortsusp}}
            \put(3600,0){\usebox\dynkinatwo}
            \put(5400,0){\usebox\shortsusp}
            \put(3400,-300){$\btl$}
            \put(5199,-300){$\btr$} \put(5000,600){\tiny $\a_d$}
          \end{picture}}
          &$1 \le d \le n$ &\ref{sss:1}\\
          \hline
          $\sD_{n\ge4}^{(1)}$
          &$(dd)$&\W\a_0+\a_1+2\a_2+\ldots+2\a_{d-1}|2\a_d+\ldots+
          2\a_{n-2}+\a_{n-1}+\a_n|&1&
          $\2$&\boxit{\begin{picture}(11550,2600)(-1300,-1300)
            \thicklines
            \put(5000,600){\tiny$\a_d$} \put(0,0){\usebox{\vertex}}
            \multiput(-1200,1200)(0,-2400){2}{\usebox{\vertex}}
            \put(-1200,-1200){\line(1,1){1200}}
            \put(-1200,1200){\line(1,-1){1200}}
            \put(0,0){\usebox{\shortsusp}}
            \put(1800,0){\usebox{\dynkinafour}}
            \put(7200,0){\usebox{\shortsusp}}
            \put(9000,0){\usebox{\bifurc}}
            \multiput(3600,0)(1800,0){2}{\circle*{600}}
          \end{picture}}&\hbox to 0pt{$2\le d \le n-1$\hss}&\ref{sss:1}/\ref{sss:2}\\
          &$(dd')$&\W2\a_1+2\a_2+\ldots+2\a_{n-2}+
          \a_{n-1}+\a_n|2\a_0+2\a_2+\ldots+2\a_{n-2}+\a_{n-1}+\a_n|&2&$\2$&
          \boxit{\begin{picture}(6400,3000)(-1500,-1500)
            \thicklines
            \put(0,0){\usebox{\vertex}}
            \multiput(-1200,1200)(0,-2400){2}{\usebox{\vertex}}
            \put(-1200,-1200){\line(1,1){1200}}
            \put(-1200,1200){\line(1,-1){1200}}
            \put(000,0){\usebox{\susp}} \put(3600,0){\usebox{\bifurc}}
            \multiput(-1200,1200)(0,-2400){2}{\circle*{600}}
          \end{picture}} &&\ref{ss:nl}\\
          &$(aa)$&\W\a_1+\a_2+
          \ldots+\a_{n-2}+\a_{n-1}|\a_0+\a_2+\ldots+\a_{n-2}+
          \a_n|&1&&\boxit{\begin{picture}(6600,3000)(-1500,-1500)
            \multiput(-1200,1200)(0,-2400){2}{\circle{600}}
            \multiput(4800,1200)(0,-2400){2}{\circle{600}}
            \multiput(-1200,-1200)(25,0){20}{\circle*{70}}
            \multiput(-700,-1200)(0,25){12}{\circle*{70}}
            \multiput(-700,-900)(25,0){12}{\circle*{70}}
            \multiput(-400,-900)(0,25){12}{\circle*{70}}
            \multiput(-400,-600)(25,0){12}{\circle*{70}}
            \multiput(-100,-600)(0,25){12}{\circle*{70}}
            \multiput(-100,-300)(25,0){12}{\circle*{70}}
            \multiput(200,-300)(25,25){8}{\circle*{70}}
            \multiput(400,-100)(25,-25){8}{\circle*{70}}
            \multiput(600,-300)(25,25){8}{\circle*{70}}
            \multiput(800,-100)(25,-25){8}{\circle*{70}}
            \multiput(1000,-300)(25,25){8}{\circle*{70}}
            \multiput(1200,-100)(25,-25){8}{\circle*{70}}
            \multiput(1400,-300)(25,25){8}{\circle*{70}}
            \multiput(1600,-100)(25,-25){8}{\circle*{70}}
            \multiput(1800,-300)(25,25){8}{\circle*{70}}
            \multiput(2000,-100)(25,-25){8}{\circle*{70}}
            \multiput(2200,-300)(25,25){8}{\circle*{70}}
            \multiput(2400,-100)(25,-25){8}{\circle*{70}}
            \multiput(2600,-300)(25,25){8}{\circle*{70}}
            \multiput(2800,-100)(25,-25){8}{\circle*{70}}
            \multiput(3000,-300)(25,25){8}{\circle*{70}}
            \multiput(3200,-100)(25,-25){8}{\circle*{70}}
            \multiput(3400,-300)(25,0){12}{\circle*{70}}
            \multiput(3700,-300)(0,-25){12}{\circle*{70}}
            \multiput(3700,-600)(25,0){12}{\circle*{70}}
            \multiput(4000,-600)(0,-25){12}{\circle*{70}}
            \multiput(4000,-900)(25,0){12}{\circle*{70}}
            \multiput(4300,-900)(0,-25){12}{\circle*{70}}
            \multiput(4300,-1200)(25,0){20}{\circle*{70}}
            \multiput(-1200,+1200)(25,0){20}{\circle*{70}}
            \multiput(-700,+1200)(0,-25){12}{\circle*{70}}
            \multiput(-700,+900)(25,0){12}{\circle*{70}}
            \multiput(-400,900)(0,-25){12}{\circle*{70}}
            \multiput(-400,600)(25,0){12}{\circle*{70}}
            \multiput(-100,600)(0,-25){12}{\circle*{70}}
            \multiput(-100,300)(25,0){12}{\circle*{70}}
            \multiput(200,300)(25,-25){8}{\circle*{70}}
            \multiput(400,100)(25,25){8}{\circle*{70}}
            \multiput(600,300)(25,-25){8}{\circle*{70}}
            \multiput(800,100)(25,25){8}{\circle*{70}}
            \multiput(1000,300)(25,-25){8}{\circle*{70}}
            \multiput(1200,100)(25,25){8}{\circle*{70}}
            \multiput(1400,300)(25,-25){8}{\circle*{70}}
            \multiput(1600,100)(25,25){8}{\circle*{70}}
            \multiput(1800,300)(25,-25){8}{\circle*{70}}
            \multiput(2000,100)(25,25){8}{\circle*{70}}
            \multiput(2200,300)(25,-25){8}{\circle*{70}}
            \multiput(2400,100)(25,25){8}{\circle*{70}}
            \multiput(2600,300)(25,-25){8}{\circle*{70}}
            \multiput(2800,100)(25,25){8}{\circle*{70}}
            \multiput(3000,300)(25,-25){8}{\circle*{70}}
            \multiput(3200,100)(25,25){8}{\circle*{70}}
            \multiput(3400,300)(25,0){12}{\circle*{70}}
            \multiput(3700,300)(0,25){12}{\circle*{70}}
            \multiput(3700,600)(25,0){12}{\circle*{70}}
            \multiput(4000,600)(0,25){12}{\circle*{70}}
            \multiput(4000,900)(25,0){12}{\circle*{70}}
            \multiput(4300,900)(0,25){12}{\circle*{70}}
            \multiput(4300,1200)(25,0){20}{\circle*{70}}
            \thicklines
            \put(0,0){\usebox{\vertex}}
            \multiput(-1200,1200)(0,-2400){2}{\usebox{\vertex}}
            \put(-1200,-1200){\line(1,1){1200}}
            \put(-1200,1200){\line(1,-1){1200}}
            \put(000,0){\usebox{\susp}} \put(3600,0){\usebox{\bifurc}}
          \end{picture}} &&\ref{ss:nl}\\
          \hline
          $\sF_4^{(1)}$&$(bf)$&\W\a_0+\a_1+\a_2+\a_3|\a_1+2\a_2+3\a_3+2\a_4|&1&&
          \boxit{\begin{picture}(7800,600)(-300,-300)
            \put(0,0){\usebox{\dynkinatwo}}
            \put(1800,0){\usebox{\dynkinffour}}
            \put(0,0){\circle*{600}} \put(7200,0){\circle*{600}}
          \end{picture}}
          & &\ref{ss:nl}\\
          &$(cd)$&\W\2\a_0+\a_1+\2\a_2|\a_2+2\a_3+\a_4|&$\2$&&
          \boxit{\begin{picture}(7400,1500)(-100,-300)
            \put(1200,600){\tiny\3} \put(0,0){\usebox{\dynkinatwo}}
            \put(1800,0){\usebox{\dynkinffour}}
            \put(1800,0){\circle*{600}} \put(5400,0){\circle*{600}}
          \end{picture}}
          & &\ref{ss:nl}\\
          &$(a\Ca)$&\WW\a_3+\a_4|\pi_2-\pi_3-\pi_4|&&
          &\boxit{\begin{picture}(7600,600)(-100,-300)
            \put(0,0){\usebox{\dynkinatwo}}
            \put(1800,0){\usebox{\dynkinffour}}
            \put(5400,0){\usebox\atwo} \put(3300,-300){$\btl$}
          \end{picture}}
          & &\ref{sss:1}\\
          \hline
          $\sG_2^{(1)}$&$(g\Ca)$&\WW\a_2+2\a_1|\pi_0-\pi_1|&&&
          \boxit{\begin{picture}(4100,1000)(-250,-500)
            \put(0,0){\usebox\dynkinathree}
            \multiput(1800,200)(0,-400){2}{\line(1,0){1800}}
            \put(3300,0){\line(-1,1){500}}
            \put(3300,0){\line(-1,-1){500}}
            \put(3600,0){\circle*{600}}
            \put(-300,-300){$\btr$}
          \end{picture}}&  &\ref{ss:nl}\\
          &$(a\Ca)$&\WW\a_1|\pi_2-2\pi_1|&&&
          \boxit{\begin{picture}(4000,1800)(-100,-900)
            \put(0,0){\usebox\dynkinathree}
            \multiput(1800,200)(0,-400){2}{\line(1,0){1800}}
            \put(3300,0){\line(-1,1){500}}
            \put(3300,0){\line(-1,-1){500}} \put(3600,0){\usebox\aone}
            \put(1300,-300){$\btl$}
          \end{picture}}&  &\ref{sss:1}\\
          &$(d\Ca)$&
          \WW\2\a_0+\2\a_1|-\pi_0+\pi_2-\pi_1|&&&
          \boxit{\begin{picture}(4100,2300)(-250,-500)
            \put(0,0){\usebox\dynkinathree}
            \multiput(1800,200)(0,-400){2}{\line(1,0){1800}}
            \put(3300,0){\line(-1,1){500}}
            \put(3300,0){\line(-1,-1){500}}
            \multiput(0,0)(3600,0){2}{\circle{600}}
            \multiput(0,300)(3600,0){2}{\line(0,1){600}}
            \put(0,900){\line(1,0){3600}}
            \put(1400,-300){$\bt$} \put(1400,1200){\tiny $\3$}
          \end{picture}}&  &\ref{sss:2}\\
          \hline
          $\underset{n\ge1}{\sA_{2n}^{(2)}}$&$(ab)$&\W2\a_0+2\a_1+
          \ldots+2\a_{n-1}|\a_n|&1&&
          \boxit{\begin{picture}(7600,2000)(-100,-900)
            \multiput(0,0)(5400,0){2}{\usebox{\leftbiedge}}
            \put(1800,0){\usebox{\susp}} \put(0,0){\usebox{\vertex}}
            \put(5400,0){\circle*{600}} \put(5100,600){$\tiny 2$}
            \put(7200,0){\usebox{\aone}}
          \end{picture}}&  &\ref{sss:1}\\
          &$(b\Cc)$&\WW\a_0+\a_1+\ldots+\a_{d-1}|-\pi_{d-1}+\pi_d, \
          (-2\pi_0+\pi_1\text{ for }d=1)|&&&
          \boxit{\begin{picture}(11000,1300)(-100,-300)
            \multiput(0,0)(9000,0){2}{\usebox{\leftbiedge}}
            \put(1800,0){\usebox{\susp}}
            \put(3600,0){\usebox\dynkinatwo}
            \put(5400,0){\usebox\susp} \put(0,0){\usebox{\vertex}}
            \put(3600,0){\circle*{600}}
            \put(5300,-300){$\btr$} \put(5000,600){\tiny $\a_d$}
          \end{picture}} &$1\le d\le n$ &\ref{sss:1}\\
          \hline
          $\underset{n\ge3}{\sA_{2n-1}^{(2)}}$&$(a_1c)$&\W\a_1+2\a_2+
          \ldots+2\a_{n-1}+\a_n|\a_0|&1&&
          \boxit{\begin{picture}(5800,2800)(-2100,-1900)
            \put(-1800,0){\usebox{\edge}} \put(0,0){\usebox{\vedge}}
            \put(0,0){\usebox{\shortsusp}}
            \put(1800,0){\usebox{\leftbiedge}}
            \put(-1800,0){\usebox{\aone}} \put(0,0){\circle*{600}}
          \end{picture}}
          & &\ref{sss:1}\\
          &$(a_3c)$&\W\a_0+\a_1+\a_2|\a_2+2\a_3+\ldots+2\a_{n-1}+\a_n|&1&&
          \boxit{\begin{picture}(7600,2400)(-2100,-2100)
            \put(-1800,0){\usebox{\edge}} \put(0,0){\usebox{\vedge}}
            \put(0,0){\usebox{\edge}}
            \put(1800,0){\usebox{\shortsusp}}
            \put(3600,0){\usebox{\leftbiedge}}
            \put(-1800,0){\circle{600}} \put(0,-1800){\circle{600}}
            \put(1800,0){\circle*{600}}
            \multiput(-1800,0)(25,-25){13}{\circle*{70}}
            \multiput(-1475,-325)(25,25){7}{\circle*{70}}
            \multiput(-1300,-150)(25,-25){7}{\circle*{70}}
            \multiput(-1125,-325)(25,25){7}{\circle*{70}}
            \multiput(-950,-150)(25,-25){7}{\circle*{70}}
            \multiput(-775,-325)(25,25){7}{\circle*{70}}
            \multiput(-600,-150)(25,-25){7}{\circle*{70}}
            \multiput(-425,-325)(25,25){7}{\circle*{70}}
            \multiput(-250,-150)(25,-25){7}{\circle*{70}}
            \multiput(-75,-325)(-25,-25){7}{\circle*{70}}
            \multiput(-250,-500)(25,-25){7}{\circle*{70}}
            \multiput(-75,-675)(-25,-25){7}{\circle*{70}}
            \multiput(-250,-850)(25,-25){7}{\circle*{70}}
            \multiput(-75,-1025)(-25,-25){7}{\circle*{70}}
            \multiput(-250,-1200)(25,-25){7}{\circle*{70}}
            \multiput(-75,-1375)(-25,-25){7}{\circle*{70}}
            \multiput(-250,-1550)(25,-25){12}{\circle*{70}}
          \end{picture}}& &\ref{ss:nl}\\
          &$(ad)$&\W\a_0+\a_1+2\a_2+\ldots+2\a_{n-1}|\a_n|&1&&
          \boxit{\begin{picture}(5800,2800)(-1900,-1900)
            \put(-1800,0){\usebox{\edge}} \put(0,0){\usebox{\vedge}}
            \put(0,0){\usebox{\shortsusp}}
            \put(1800,0){\usebox{\leftbiedge}}
            \put(1800,0){\circle*{600}} \put(3600,0){\usebox{\aone}}
          \end{picture}}&  &\ref{sss:1}\\
          &$(d\Cc)$&\WW\2\a_0+\2\a_1+\a_2+\ldots+\a_{d-1}|
          -\pi_{d-1}+\pi_d,\ (-\pi_0-\pi_1+\pi_2\text{ for }d=2)
          |&&&\boxit{\begin{picture}(11000,3100)(-1900,-1900)
            \put(-1800,0){\usebox{\edge}} \put(0,0){\usebox{\vedge}}
            \put(0,0){\usebox{\edge}}
            \put(1800,0){\usebox{\shortsusp}}
            \put(3600,0){\usebox{\dynkinatwo}}
            \put(5400,0){\usebox{\shortsusp}}
            \put(7200,0){\usebox{\leftbiedge}}
            \put(3600,0){\circle*{600}}
            \put(5100,-300){$\btr$} \put(3000,600){\tiny
              $\3$} \put(5000,600){\tiny $\a_d$}
          \end{picture}}&$2\le d \le n$ &\ref{sss:1}/\ref{sss:2}\\
          &$(\Cc\Cc)$&\WWW\pi_0-\pi_1|-\pi_0+\pi_1|&&&
          \boxit{\begin{picture}(9200,2400)(-2000,-2100)
            \put(-1800,0){\usebox{\edge}} \put(0,0){\usebox{\vedge}}
            \put(0,0){\usebox{\edge}} \put(1800,0){\usebox{\susp}}
            \put(5400,0){\usebox{\leftbiedge}}
            \put(-2100,-300){$\btr$} \put(-400,-2100){$\bt$}
          \end{picture}}& &\ref{ss:nl}\\
         \hline
          $\underset{n\ge2}{\sD_{n+1}^{(2)}}$
          &$(bb)$&\W\a_0+\ldots+\a_{d-1}|\a_d+\ldots+\a_n|&1&$2$&
          \boxit{\begin{picture}(9200,1300)(-100,-300)
            \put(0,0){\usebox{\vertex}}
            \put(0,0){\usebox{\leftbiedge}}
            \put(7200,0){\usebox{\rightbiedge}}
            \multiput(1800,0)(3600,0){2}{\usebox{\shortsusp}}
            \put(3600,0){\usebox{\dynkinatwo}}
            \put(3600,0){\circle*{600}} \put(5400,0){\circle*{600}}
            \put(5000,600){\tiny $\a_d$}
          \end{picture}}
          &  $1\le d\le n$   &\ref{sss:1}\\
          &$(ad)$&\W\a_1+\ldots+\a_{n-1}|\a_0+\a_n|&1&$[\2]_{n=2}$&
          \boxit{\begin{picture}(7800,1500)(-300,-300)
            \put(0,0){\usebox{\vertex}}
            \put(0,0){\usebox{\leftbiedge}}
            \put(5400,0){\usebox{\rightbiedge}}
            \put(1800,0){\usebox{\shortam}}
            \multiput(0,0)(7200,0){2}{\usebox{\wcircle}}
            \multiput(0,250)(7200,0){2}{\line(0,1){950}}
            \put(0,1200){\line(1,0){7200}}
          \end{picture}}
          &  &\ref{sss:2}\\
          &$(dd)$&\W\a_0+\a_2|\a_1+\a_3|&1&$\2$ &
          \boxit{\begin{picture}(6000,2400)(-300,-1200)
            \put(0,0){\usebox{\vertex}}
            \put(0,0){\usebox{\leftbiedge}}
            \put(3600,0){\usebox{\rightbiedge}}
            \put(1800,0){\usebox{\dynkinatwo}}
            \multiput(0,0)(3600,0){2}{\usebox{\wcircle}}
            \multiput(1800,0)(3600,0){2}{\usebox{\wcircle}}
            \multiput(0,-250)(3600,0){2}{\line(0,-1){950}}
            \multiput(1800,250)(3600,0){2}{\line(0,1){950}}
            \put(0,-1200){\line(1,0){3600}}
            \put(1800,1200){\line(1,0){3600}}
          \end{picture}}&   &\ref{sss:3} \\
          &$(b'b')$&\W3\a_0+2\a_1+\a_2|\a_1+2\a_2+3\a_3|&3&$\2$&
          \boxit{\begin{picture}(6000,600)(-300,-300)
            \put(0,0){\usebox{\vertex}}
            \put(0,0){\usebox{\leftbiedge}}
            \put(3600,0){\usebox{\rightbiedge}}
            \put(1800,0){\usebox{\dynkinatwo}}
            \multiput(0,0)(5400,0){2}{\circle*{600}}
          \end{picture}}& &\ref{ss:nl}\\
          &$(d\Ca)$&\WW\2\a_0+\2\a_2|-\pi_0+\pi_1-\pi_2|&&&
          \boxit{\begin{picture}(4200,1900)(-300,-300)
            \put(0,0){\usebox{\vertex}}
            \put(0,0){\usebox{\leftbiedge}}
            \put(1800,0){\usebox{\rightbiedge}}
            \multiput(0,0)(3600,0){2}{\circle{600}}
            \multiput(0,300)(3600,0){2}{\line(0,1){300}}
            \put(0,600){\line(1,0){3600}} \put(1500,1000){\tiny
              $\3$} \put(1500,-300){$\bt$}
          \end{picture}} &  &\ref{sss:2}\\
          &$(\Cc\Cc)$&\WWW\pi_0-\pi_2|-\pi_0+\pi_2|&&&
          \boxit{\begin{picture}(4200,600)(-300,-300)
            \put(0,0){\usebox{\vertex}}
            \put(0,0){\usebox{\leftbiedge}}
            \put(1800,0){\usebox{\rightbiedge}}
            \put(-400,-300){$\btr$} \put(3300,-300){$\btl$}
          \end{picture}}&  &\ref{ss:nl} \\
          \hline
          $\sE_{6}^{(2)}$&$(af)$&\W\a_0|2\a_1+3\a_2+2\a_3+\a_4|&1&&
          \boxit{\begin{picture}(7600,1800)(-300,-900)
            \put(0,0){\usebox{\dynkinathree}}
            \put(3600,0){\usebox{\leftbiedge}}
            \put(5400,0){\usebox{\dynkinatwo}}
            \put(0,0){\usebox{\aone}} \put(1800,0){\circle*{600}}
          \end{picture}}
          &   &\ref{sss:1} \\
          &$(bc)$&\W\a_0+2\a_1+2\a_2+\a_3|\a_2+\a_3+\a_4|&1&&
          \boxit{\begin{picture}(7600,600)(-100,-300)
            \put(0,0){\usebox{\dynkinathree}}
            \put(3600,0){\usebox{\leftbiedge}}
            \put(5400,0){\usebox{\dynkinatwo}}
            \put(7200,0){\circle*{600}} \put(1800,0){\circle*{600}}
          \end{picture}}
          & &\ref{ss:nl}\\
          &$(a\Ca)$&\WW\a_0+\a_1+\a_2|-\pi_0-\pi_2+\pi_3|&&&
          \boxit{\begin{picture}(7600,600)(-300,-300)
            \put(0,0){\usebox{\athree}}
            \put(3600,0){\usebox{\leftbiedge}}
            \put(5400,0){\usebox{\dynkinatwo}} \put(5100,-300){$\btr$}
          \end{picture}}
          &  &\ref{sss:1}\\
          \hline
          $\sD_4^{(3)}$&$(ag)$&\W\a_0|2\a_1+\a_2|&1&$2$&
          \boxit{\begin{picture}(4000,1800)(-300,-900)
            \put(0,0){\usebox{\dynkinatwo}}
            \put(1800,0){\usebox{\dynkingtwo}}
            \put(0,0){\usebox{\aone}} \put(1800,0){\circle*{600}}
          \end{picture}}
          &   &\ref{sss:1} \\
          &$(ad)$&\W\a_0+\a_2|2\a_1|&1&$\2$&
          \boxit{\begin{picture}(4200,1500)(-300,-1200)
            \put(0,0){\usebox{\dynkinatwo}}
            \put(1800,0){\usebox{\dynkingtwo}}
            \put(1800,0){\usebox{\aprime}}
            \multiput(0,0)(3600,0){2}{\usebox{\wcircle}}
            \multiput(0,-250)(3600,0){2}{\line(0,-1){950}}
            \put(0,-1200){\line(1,0){3600}}
          \end{picture}}
          &  &\ref{sss:2}\\
          &$(a\Ca)$&\WW\a_0+\a_1|-\pi_0-\pi_1+\pi_2|&&&
          \boxit{\begin{picture}(4200,600)(-300,-300)
            \put(0,0){\usebox{\atwo}}
            \put(1800,0){\usebox{\dynkingtwo}} \put(3300,-300){$\btr$}
  \end{picture}}
  &  &\ref{sss:1}\\
  \hline
\noalign{\begin{center}\bf The finite simple cases\end{center}}
\hline
$\sA_{n\ge1}$&$(a\Ca)$&\WW\a_1+\ldots+\a_{d-1}|-\pi_1-\pi_{d-1}+\pi_d|&&&\boxit{\begin{picture}(9400,1250)(-300,-300)
\put(0,0){\usebox\shortam}
\put(3600,0){\usebox\edge}
\put(5400,0){\usebox{\shortsusp}}
\put(7200,0){\usebox\dynkinatwo}
\put(5100,-300){$\btr$}
\put(5100,600){\tiny$\a_d$}
\end{picture}}
&$2\le d\le n$ &\ref{sss:1}
\\
&$(d_2\Ca)$&\WW\2\a_1+\2\a_3|-\pi_1+\pi_2-\pi_3|&&&
\boxit{\begin{picture}(4200,2200)(-300,-300)
\put(1500,1200){\tiny $\3$}
\put(0,0){\usebox\dynkinathree}
\multiput(0,0)(3600,0){2}{\circle{600}}
\multiput(0,300)(3600,0){2}{\line(0,1){600}}
\put(0,900){\line(1,0){3600}}
\put(1600,-300){$\bt$}
 \end{picture}}
 & &\ref{sss:2}\\
&$(d_3\Ca)$&\WW\a_1+2\a_2+\a_3|-2\pi_2+\pi_4|&&&
\boxit{\begin{picture}(5900,600)(-100,-300)
\put(0,0){\usebox\dynkinafour}
\put(1800,0){\circle*{600}}
\put(5200,-300){$\btl$}
\end{picture}}
& &\ref{sss:1}
\\
&$(\Ca\Ca)$&\WWW\pi_{d-1}-\pi_d|-\pi_{d-1}+\pi_d|&&&
\boxit{\begin{picture}(9200,1250)(-100,-300)
\put(0,0){\usebox\susp}
\put(3600,0){\usebox\dynkinatwo}
\put(5400,0){\usebox\susp}
\put(3300,-300){$\btl$}
\put(5100,-300){$\btr$}
\put(5100,600){\tiny$\a_d$}
 \end{picture}}
&$2\le d\le n$ &\ref{sss:1}
\\
&$(\Ca\Ca')$&\WWW\pi_1-\pi_n|-\pi_1+\pi_n|&&&
\boxit{\begin{picture}(9500,600)(-300,-300)
  \thicklines
    \multiput(0,0)(5400,0){2}{\usebox{\dynkinathree}}
    \put(3600,0){\usebox{\shortsusp}}
    \put(-350,-300){$\btr$}
    \put(8600, -300){$\btl$}
\end{picture}}
&$n\ge2$ &\ref{ss:nl}
\\
&$(a)$&\WW\a_1+\ldots+\a_n||&&$[2]_{n=1}$&
\boxit{\begin{picture}(6000,600)(-300,-300)
    \put(0,0){\usebox{\mediumam}}
  \end{picture}}
&&\ref{table:local}
\\
&$(d)$&\WW\a_1+2\a_2+\a_3||&&$\2$&
\boxit{\begin{picture}(3800,600)(-100,-300)
    \put(0,0){\usebox{\dynkinathree}}
        \put(1800,0){\circle*{600}}
      \end{picture}}&&\ref{table:local}
\\
      &$(\Ca)$&\WWW\pi_1||&&&
      \boxit{\begin{picture}(11200,600)(-300,-300)
\put(0,0){\usebox\dynkinathree}
\put(3600,0){\usebox\susp}
\put(7200,0){\usebox\dynkinathree}
\put(-350,-300){$\btr$}
\end{picture}}&&\ref{table:local}
\\
\hline
$\sB_{n\ge2}$&$(b\Ca)$&\WW\a_d+\ldots+\a_n|\pi_{d-1}-\pi_d|&&&
\boxit{\begin{picture}(9200,1250)(-100,-300)
\put(0,0){\usebox\shortsusp}
\put(1800,0){\usebox\dynkinatwo}
\put(3600,0){\usebox\shortsusp}
\put(5400,0){\usebox\dynkinbthree}
\put(3600,0){\circle*{600}}
\put(3200,600){\tiny$\a_d$}
\put(1500,-300){$\btl$}
\end{picture}}
&$2\le d\le n$ &\ref{sss:1}
\\
&$(d\Ca)$&\WW\2\a_1+\2\a_3|-\pi_1+\pi_2-\pi_3|&&&
\boxit{\begin{picture}(4200,2150)(-300,-300)
\put(1200,1200){\tiny $\3$}
\put(0,0){\usebox\dynkinbthree}
\multiput(0,0)(3600,0){2}{\circle{600}}
\multiput(0,300)(3600,0){2}{\line(0,1){600}}
\put(0,900){\line(1,0){3600}}
\put(1500,-300){$\bt$}
 \end{picture}}
& &\ref{sss:2}
\\
&$(d\Cc)$&\WW\a_1+2\a_2+\a_3|-2\pi_2+\pi_4|&&&
\boxit{\begin{picture}(5800,1300)(-100,-300)
\put(0,0){\usebox\dynkinbfour}
\put(1800,0){\circle*{600}}
\put(5100,-300){$\btl$}
\put(1300,400){\tiny $\3$}
\end{picture}}
& &\ref{ss:nl}
\\
&$(b'\Ca)$&\WW\a_2+2\a_3+3\a_4|\pi_1-2\pi_4|&&&
\boxit{\begin{picture}(6000,600)(-300,-300)
\put(0,0){\usebox\dynkinbfour}
\put(5400,0){\circle*{600}}
\put(-400,-300){$\btr$}
\end{picture}}
& &\ref{ss:nl}
\\
&$(\Ca\Ca)$&\WWW\pi_{n-1}-\pi_n|-\pi_{n-1}+\pi_n|&&&
\boxit{\begin{picture}(7600,600)(-100,-300)
\put(0,0){\usebox\dynkinatwo}
\put(1800,0){\usebox\shortsusp}
\put(3600,0){\usebox\dynkinbthree}
\put(5100,-300){$\btl$}
\put(6900,-300){$\btr$}
\end{picture}}
& &\ref{sss:1}
\\
&$(\Ca\Cc)$&\WWW\pi_1-\pi_3|-\pi_1+\pi_3|&&&
\boxit{\begin{picture}(4200,600)(-300,-300)
\put(0,0){\usebox\dynkinbthree}
\put(3300,-300){$\btl$}
\put(-400,-300){$\btr$}
 \end{picture}}
& &\ref{ss:nl}
\\
&$(b)$&\WW\a_1+\ldots+\a_n||&&$2$&
\boxit{\begin{picture}(7600,600)(-300,-300)
    \put(0,0){\usebox{\shortbm}}
    \end{picture}}
&&\ref{table:local}
\\
&$(b')$&\WW\a_1+2\a_2+3\a_3||&&$\2$&
\boxit{\begin{picture}(4000,600)(-100,-300)
    \put(0,0){\usebox{\bthirdthree}}
    \end{picture}}
  &&\ref{table:local}
\\
\hline
$\sC_{n\ge3}$&$(a\Cc)$&\WW\a_1+\ldots+\a_{d-1}|-\pi_1-\pi_{d-1}+\pi_d|&&&\boxit{\begin{picture}(9400,1250)(-300,-300)
\put(0,0){\usebox\shortam}
\put(3600,0){\usebox\dynkinatwo}
\put(5400,0){\usebox\shortsusp}
\put(7200,0){\usebox\leftbiedge}
\put(5100,-300){$\btr$}
\put(5100,600){\tiny$\a_d$}
\end{picture}}
&$2\le d\le n$ &\ref{sss:1}\\  
&$(c\Ca)$&\WW\a_2+2\a_3+\ldots+2\a_{n-1}+\a_n|\pi_1-\pi_3|&&&
\boxit{\begin{picture}(9400,600)(-300,-300)
\put(0,0){\usebox\dynkinathree}
\put(3600,0){\usebox\shortsusp}
\put(5400,0){\usebox\dynkincthree}
\put(-400,-300){$\btr$}
\put(3600,0){\circle*{600}}
\end{picture}}
& &\ref{ss:nl}\\
&$(\Ca\Cc)$&\WWW\pi_{d-1}-\pi_d|-\pi_{d-1}+\pi_d|&&&
\boxit{\begin{picture}(11000,1250)(-100,-300)
    \put(0,0){\usebox\dynkinatwo}
    \put(1800,0){\usebox\shortsusp}
\put(3600,0){\usebox\dynkinatwo}
\put(5400,0){\usebox\shortsusp}
\put(7200,0){\usebox\dynkincthree}
\put(5100,-300){$\btr$}
\put(5100,600){\tiny$\a_d$}
\put(3300,-300){$\btl$}
\end{picture}}
&$2\le d\le n$ &\ref{sss:1}
\\
&$(c)$&\WW\a_1+2\a_2+\ldots+2\a_{n-1}+\a_n||&&&
\boxit{\begin{picture}(9200,600)(-100,-300)
    \put(0,0){\usebox{\shortcm}}
  \end{picture}}
&&\ref{table:local}
\\
&$(\Cc)$&\WWW\pi_1||&&&
\boxit{\begin{picture}(11200,600)(-300,-300)
\put(0,0){\usebox\dynkinathree}
\put(3600,0){\usebox\susp}
\put(7200,0){\usebox\dynkincthree}
\put(-400,-300){$\btr$}
\end{picture}}
&&\ref{table:local}
\\
\hline
$\sD_{n\ge4}$&$(d\Ca)$&\WW\a_d+\ldots+\a_{n-2}+\2\a_{n-1}+\2\a_n|\pi_{d-1}-\pi_d,\ (\pi_{n-2}-\pi_{n-1}-\pi_n)|&&&
\boxit{\begin{picture}(8600,2600)(-100,-1300)
\put(0,0){\usebox\shortsusp}
\put(1800,0){\usebox\dynkinatwo}
\put(3600,0){\usebox\shortsusp}
\put(5400,0){\usebox\dynkindfour}
\put(3600,0){\circle*{600}}
\put(3100,600){\tiny$\3$}
\put(3300,-1000){\tiny$\a_d$}
\put(1500,-300){$\btl$}
\end{picture}}
&$2\le d<n$ &\ref{sss:1}/\ref{sss:2}
\\
&$(a\Ca)$&\WW\a_1+\ldots+\a_{n-1}|-\pi_1-\pi_{n-1}+\pi_n|&&&
\boxit{\begin{picture}(7200,3000)(-300,-1500)
\put(0,0){\usebox\amne}
\put(6300,-1550){$\btlu$}
\end{picture}}
& &\ref{ss:nl}
\\
&$(d_4\Ca)$&\WW\a_2+2\a_3+\a_4+2\a_5|\pi_1-2\pi_5|&&&
\boxit{\begin{picture}(5400,2800)(-300,-1500)
\put(0,0){\usebox\dynkinatwo}
\put(1800,0){\usebox\dynkindfour}
\put(-400,-300){$\btr$}
\put(4800,-1200){\circle*{600}}
\end{picture}}
& &\ref{ss:nl}
\\
&$(\Ca\Ca)$&\WWW\pi_{n-1}-\pi_n|-\pi_{n-1}+\pi_n|&&&
\boxit{\begin{picture}(7000,3000)(-100,-1500)
\put(0,0){\usebox\susp}
\put(3600,0){\usebox\dynkindfour}
\put(6300,950){$\btld$}
\put(6300,-1550){$\btlu$}
\end{picture}}
& &\ref{ss:nl}
\\
&$(d)$&\WW 2\a_1+\ldots+2\a_{n-2}+\a_{n-1}+\a_n||&&$\2$&
\boxit{\begin{picture}(7000,2600)(-300,-1300)
\put(0,0){\usebox{\shortdm}}
\end{picture}}&&\ref{table:local}
\\
\hline
$\sF_4$&$(c\Ca)$&\WW\a_2+2\a_3+\a_3|\pi_1-\pi_3|&&&
\boxit{\begin{picture}(5800,600)(-300,-300)
\put(0,0){\usebox\dynkinffour}
\put(3600,0){\circle*{600}}
\put(-400,-300){$\btr$}
\end{picture}}
& &\ref{ss:nl}\\
&$(a\Ca)$&\WW\a_3+\a_4|\pi_2-\pi_3-\pi_4|&&&
\boxit{\begin{picture}(5800,600)(-100,-300)
\put(0,0){\usebox\dynkinffour}
\put(3600,0){\usebox\atwo}
\put(1500,-300){$\btl$}
\end{picture}}
& &\ref{sss:1}\\
&$(b\Cc)$&\WW\a_1+\a_2+\a_3|-\pi_1+\pi_4|&&&
\boxit{\begin{picture}(6000,600)(-300,-300)
\put(0,0){\usebox\dynkinffour}
\put(0,0){\circle*{600}}
\put(5100,-300){$\btl$}
\end{picture}}
& &\ref{ss:nl}\\
&$(\Ca\Ca)$&\WWW\pi_2-\pi_3|-\pi_2+\pi_3|&&&
\boxit{\begin{picture}(5600,600)(-100,-300)
\put(0,0){\usebox\dynkinffour}
\put(1500,-300){$\btl$}
\put(3300,-300){$\btr$}
\end{picture}}
& &\ref{sss:1}\\
&$(f)$&\WW\a_1+2\a_2+3\a_3+2\a_4||&&&
\boxit{\begin{picture}(5800,600)(-100,-300)
\put(0,0){\usebox{\ffour}}
\end{picture}}
&&\ref{table:local}
\\
\hline
$\sG_2$&$(a\Ca)$&\WW\a_1|-2\pi_1+\pi_2|&&&
\boxit{\begin{picture}(2400,1800)(-300,-900)
\put(0,0){\usebox\dynkingtwo}
\put(0,0){\usebox\aone}
\put(1500,-300){$\btr$}
\end{picture}}
& &\ref{sss:1}\\
&$(\Ca\Ca)$&\WWW\pi_1-\pi_2|-\pi_1+\pi_2|&&&
\boxit{\begin{picture}(2400,600)(-300,-300)
\put(0,0){\usebox\dynkingtwo}
\put(1500,-300){$\btr$}
\put(-400,-300){$\btl$}
\end{picture}}
& &\ref{sss:1}
\\
&$(g)$&\WW2\a_1+\a_2||&&$2$&
\boxit{\begin{picture}(2200,600)(-300,-300)
\put(0,0){\usebox{\gtwo}}
\end{picture}}
&&\ref{table:local}\\
\hline
\noalign{\begin{center}\bf The reducible cases\end{center}}
\hline
$\sA_1^{(1)}\times \sA_1^{(1)}$&$(dd)$&\W\a_0+\a_0'|\a_1+\a_1'|&&$\2$&
      \boxit{\begin{picture}(5100,1800)(-300,-900)
        \put(0,0){\usebox\leftrightbiedge}
        \put(2700,0){\usebox\leftrightbiedge}
        \multiput(1800,0)(2700,0){2}{\usebox\wcircle}
        \multiput(1800,300)(2700,0){2}{\line(0,1){600}}
        \put(1800,900){\line(1,0){2700}}
        \multiput(0,0)(2700,0){2}{\usebox\wcircle}
        \multiput(0,-300)(2700,0){2}{\line(0,-1){600}}
        \put(0,-900){\line(1,0){2700}}
      \end{picture}}
      &$\frac\d{\|\a_i\|}{=}\frac{\d'}{\|\a_i'\|}$ &\ref{sss:2}\\
$\sA_2^{(2)}\times \sA_2^{(2)}$&$(dd)$&\W\a_0+\a_0'|\2\a_1+\2\a_1'|&&&
      \boxit{\begin{picture}(5100,2750)(-300,-900)
        \multiput(2700,0)(1800,0){2}{\circle*{300}}
        \multiput(2700,-60)(0,120){2}{\line(1,0){1800}}
        \multiput(2700,-180)(0,360){2}{\line(1,0){1800}}
        \multiput(2850,0)(25,25){20}{\circle*{50}}
        \multiput(2850,0)(25,-25){20}{\circle*{50}}
        \multiput(0,0)(1800,0){2}{\circle*{300}}
        \multiput(0,-60)(0,120){2}{\line(1,0){1800}}
        \multiput(0,-180)(0,360){2}{\line(1,0){1800}}
        \multiput(150,0)(25,25){20}{\circle*{50}}
        \multiput(150,0)(25,-25){20}{\circle*{50}}
        \multiput(1800,0)(2700,0){2}{\usebox\wcircle}
        \multiput(1800,300)(2700,0){2}{\line(0,1){600}}
        \put(1800,900){\line(1,0){2700}}
        \multiput(0,0)(2700,0){2}{\usebox\wcircle}
        \multiput(0,-300)(2700,0){2}{\line(0,-1){600}}
        \put(0,-900){\line(1,0){2700}}
        \put(2700,1200){\tiny $\3$}
      \end{picture}}
      &$\frac\d{\|\a_i\|}{=}\frac{\d'}{\|\a_i'\|}$ &\ref{sss:2}\\
      \hline
$\sA_1\times \sA_1^{(1)}$&$(d\Ca)$&\WW\2\a_1+\2\a_0'|-\pi_1-\pi_0'+\pi_1'|&&&
      \boxit{\begin{picture}(5100,1200)(-300,-900)
        \put(0,0){\usebox{\vertex}}
        \put (2700,0){\usebox\leftrightbiedge}
        \put(4200,-300){$\bt$}
        \multiput(0,0)(2700,0){2}{\usebox\wcircle}
        \multiput(0,-300)(2700,0){2}{\line(0,-1){600}}
        \put(0,-900){\line(1,0){2700}}
        \put(900,-600){\tiny $\3$}
      \end{picture}}
      & &\ref{sss:1}\\
$\underset{n\ge1}{\sA_1\times \sA_{2n}^{(2)}}$&$(d\Cc)$&\WW\a_1+\a_0'|-2\pi_1-2\pi_0'+\pi_1'|&&&
      \boxit{\begin{picture}(10300,1200)(-300,-900)
        \put(0,0){\usebox{\vertex}}
        \put(2700,0){\usebox{\vertex}}
        \put (2700,0){\usebox\leftbiedge}
        \put(4500,0){\usebox\susp}
        \put(8100,0){\usebox\leftbiedge}
        \put(4200,-300){$\btr$}
        \multiput(0,0)(2700,0){2}{\usebox\wcircle}
        \multiput(0,-300)(2700,0){2}{\line(0,-1){600}}
        \put(0,-900){\line(1,0){2700}}
      \end{picture}}
      & &\ref{sss:1}\\
$\sA_1\times \sC_{n\ge2}^{(1)}$&$(d\Cc)$&\WW\2\a_1+\2\a_0'|-\pi_1-\pi_0'+\pi_1'|&&&
      \boxit{\begin{picture}(10300,1200)(-300,-900)
        \put(0,0){\usebox{\vertex}}
        \put (2700,0){\usebox\rightbiedge}
        \put(4500,0){\usebox\susp}
        \put(8100,0){\usebox\leftbiedge}
        \put(4200,-300){$\btr$}
        \multiput(0,0)(2700,0){2}{\usebox\wcircle}
        \multiput(0,-300)(2700,0){2}{\line(0,-1){600}}
        \put(0,-900){\line(1,0){2700}}
        \put(900,-600){\tiny $\3$}
      \end{picture}}
      & &\ref{sss:1}\\
$\sA_1\times \sG_2^{(1)}$&$(d\Ca)$&\WW\a_1+\a_0'|-2\pi_1-2\pi_0'+\pi_1'|&&&
      \boxit{\begin{picture}(6700,1200)(-300,-900)
        \put(0,0){\usebox{\vertex}}
        \put (2700,0){\usebox\lefttriedge}
        \put(4500,0){\usebox\dynkinatwo}
        \put(4200,-300){$\btr$}
        \multiput(0,0)(2700,0){2}{\usebox\wcircle}
        \multiput(0,-300)(2700,0){2}{\line(0,-1){600}}
        \put(0,-900){\line(1,0){2700}}
      \end{picture}}
      & &\ref{sss:1}\\
      \hline
$\sA_1\times \sA_{n\ge2}$&$(d\Ca)$&\WW\a_1+\a_1'|-2\pi_1-2\pi_1'+\pi_2'|&&&
      \boxit{\begin{picture}(8500,1200)(-300,-900)
        \put(0,0){\usebox{\vertex}}
        \put (2700,0){\usebox\dynkinatwo}
        \put (4500,0){\usebox\susp}
        \put(4200,-300){$\btr$}
        \multiput(0,0)(2700,0){2}{\usebox\wcircle}
        \multiput(0,-300)(2700,0){2}{\line(0,-1){600}}
        \put(0,-900){\line(1,0){2700}}
      \end{picture}}
      & &\ref{sss:1}\\
$\sA_1\times \sC_{n\ge3}$&$(d\Cc)$&\WW\a_1+\a_1'|-2\pi_1-2\pi_1'+\pi_2'|&&&
      \boxit{\begin{picture}(10300,1200)(-300,-900)
        \put(0,0){\usebox{\vertex}}
        \put (2700,0){\usebox\dynkinatwo}
        \put (4500,0){\usebox\susp}
        \put(8100,0){\usebox\leftbiedge}
        \put(4200,-300){$\btr$}
        \multiput(0,0)(2700,0){2}{\usebox\wcircle}
        \multiput(0,-300)(2700,0){2}{\line(0,-1){600}}
        \put(0,-900){\line(1,0){2700}}
      \end{picture}}
      & &\ref{sss:1}\\
$\sA_1\times \sB_{n\ge2}$&$(d\Ca)$&\WW\a_1+\a_1'|-2\pi_1-2\pi_1'+\pi_2'|&&&
      \boxit{\begin{picture}(8500,1200)(-300,-900)
        \put(0,0){\usebox{\vertex}}
        \put(2700,0){\usebox{\vertex}}
        \put (2700,0){\usebox\leftbiedge}
        \put (4500,0){\usebox\susp}
        \put(4200,-300){$\btr$}
        \multiput(0,0)(2700,0){2}{\usebox\wcircle}
        \multiput(0,-300)(2700,0){2}{\line(0,-1){600}}
        \put(0,-900){\line(1,0){2700}}
      \end{picture}}
      & &\ref{sss:1}\\
$\sA_1\times \sC_{n\ge2}$&$(d\Ca)$&\WW\2\a_1+\2\a_1'|-\pi_1-\pi_1'+\pi_2'|&&&
      \boxit{\begin{picture}(8500,1200)(-300,-900)
        \put(0,0){\usebox{\vertex}}
        \put(2700,0){\usebox\rightbiedge}
        \put(4200,-300){$\btr$}
        \put(900,-600){\tiny $\3$}
        \multiput(0,0)(2700,0){2}{\usebox\wcircle}
        \multiput(0,-300)(2700,0){2}{\line(0,-1){600}}
        \put(0,-900){\line(1,0){2700}}
        \put (4500,0){\usebox\susp}
      \end{picture}}
      & &\ref{sss:1}\\
$\sA_1\times \sG_2$&$(d\Ca)$&\WW\a_1+\a_1'|-2\pi_1-2\pi_1'+\pi_2'|&&&
      \boxit{\begin{picture}(5100,1200)(-300,-900)
        \put(0,0){\usebox{\vertex}}
        \put(2700,0){\usebox\lefttriedge}
        \put(4200,-300){$\bt$}
        \multiput(0,0)(2700,0){2}{\usebox\wcircle}
        \multiput(0,-300)(2700,0){2}{\line(0,-1){600}}
        \put(0,-900){\line(1,0){2700}}
      \end{picture}}
      & &\ref{sss:1}\\
$\sA_{m\ge1}\times \sA_{n\ge1}$&$(\Ca\Ca)$&\WWW\pi_1-\pi_1'|-\pi_1+\pi_1'|&&&
      \boxit{$\underset{\begin{picture}(9400,600)(-300,-300)
        \put(0,0){\usebox\dynkinatwo}
        \put(1800,0){\usebox\susp}
        \put(5400,0){\usebox\dynkinathree}
        \put(-350,-300){$\btr$}
      \end{picture}}
      {\begin{picture}(9400,600)(-300,-300)
        \put(0,0){\usebox\dynkinatwo}
        \put(1800,0){\usebox\susp}
        \put(5400,0){\usebox\dynkinathree}
        \put(-350,-300){$\btr$}
      \end{picture}}$}
      & &\ref{sss:0}\\
$\sA_{m\ge1}\times \sC_{n\ge2}$&$(\Ca\Cc)$&\WWW\pi_1-\pi_1'|-\pi_1+\pi_1'|&&&
\boxit{$\underset
      {\begin{picture}(9400,600)(-300,-300)
        \put(0,0){\usebox\dynkinatwo}
        \put(1800,0){\usebox\susp}
        \put(5400,0){\usebox\dynkincthree}
        \put(-350,-300){$\btr$}
      \end{picture}}
    {\begin{picture}(9400,600)(-300,-300)
        \put(0,0){\usebox\dynkinatwo}
        \put(1800,0){\usebox\susp}
        \put(5400,0){\usebox\dynkinathree}
        \put(-350,-300){$\btr$}
      \end{picture}}$}
      & &\ref{sss:0}\\
      $\sC_{m\ge2}\times \sC_{n\ge2}$&$(\Cc\Cc)$&\WWW\pi_1-\pi_1'|-\pi_1+\pi_1'|&&&
      \boxit{$\underset{\begin{picture}(9400,600)(-300,-300)
        \put(0,0){\usebox\dynkinatwo}
        \put(1800,0){\usebox\susp}
        \put(5400,0){\usebox\dynkincthree}
        \put(-350,-300){$\btr$}
      \end{picture}}
     {\begin{picture}(9400,600)(-300,-300)
        \put(0,0){\usebox\dynkinatwo}
        \put(1800,0){\usebox\susp}
        \put(5400,0){\usebox\dynkincthree}
        \put(-350,-300){$\btr$}
      \end{picture}}$}
      & &\ref{sss:0}\\
      \hline
      $\sA_1\times \sA_1$&$(d)$&\WW\a_1+\a_1'||&&$\2$&
      \boxit{\begin{picture}(3300,1200)(-300,-900)
        \put(0,0){\usebox{\vertex}}
        \put(2700,0){\usebox{\vertex}}
        \multiput(0,0)(2700,0){2}{\usebox\wcircle}
        \multiput(0,-300)(2700,0){2}{\line(0,-1){600}}
        \put(0,-900){\line(1,0){2700}}
      \end{picture}}
      &&\ref{table:local}\\
      \hline
    \end{longtable}
  }
\end{lis}

\end{document}